\documentclass[11pt]{article}
\usepackage[utf8]{inputenc}

\usepackage{geometry} 
\usepackage{amssymb}
\usepackage{amsmath}
\usepackage{amsthm}
\usepackage{enumerate}
\usepackage{graphicx, wrapfig}
\usepackage{fullpage}
\usepackage{caption}
\usepackage{subcaption}
\usepackage{lipsum}
\usepackage{algorithm}
\usepackage{algorithmic}
\usepackage[hidelinks]{hyperref}
\usepackage{color}

\title{Monostability and bistability of biological switches
}

\author{Jules Guilberteau
\thanks{Sorbonne Universit\'e and Universit\'e de Paris, CNRS, Inria, Laboratoire Jacques-Louis Lions (LJLL), F-75005 Paris, France. {\tt\small guilberteau@ljll.math.upmc.fr}}, 
Camille Pouchol
\thanks{Universit\'e de Paris, FP2M, CNRS FR 2036, MAP5 UMR 8145, F-75006 Paris,
France. 
         {\tt\small camille.pouchol@u-paris.fr}}
\ and 
Nastassia Pouradier Duteil
\thanks{ 
       Sorbonne Universit\'e and Universit\'e de Paris, Inria, CNRS, Laboratoire Jacques-Louis Lions (LJLL), F-75005 Paris, France.  {\tt\small nastassia.pouradier\_duteil@sorbonne-universite.fr}}
}

\date{}
\newtheorem{theorem}{Theorem}
\newtheorem{prop}{Proposition}
\newtheorem{cor}[prop]{Corollary}
\newtheorem{lem}{Lemma}
\newtheorem*{theorem*}{Theorem}

\theoremstyle{definition}
\newtheorem{definition}{Definition}

\theoremstyle{definition}
\newtheorem*{rem}{Remark}
\newtheorem*{ex}{Example}

\newcommand{\R}{\mathbb{R}}

\begin{document}

\maketitle


\begin{abstract}

Cell-fate transition can be modeled by ordinary differential equations (ODEs) which describe the behavior of several molecules in interaction, and for which each stable equilibrium corresponds to a possible phenotype (or `biological trait'). In this paper, we focus on simple ODE systems modeling two molecules which each negatively (or positively) regulate the other. 
It is well-known that such models may lead to monostability or multistability, depending on the selected parameters.
However, extensive numerical simulations have led systems biologists to conjecture that in the vast majority of cases, there cannot be more than two stable points.
 Our main result is a proof of this conjecture.
 More specifically, we provide a criterion ensuring at most bistability, which is indeed satisfied by most commonly used functions.
This includes Hill functions, but also a wide family of convex and sigmoid functions. 
We also determine which parameters lead to monostability, and which lead to bistability, by developing a more general framework encompassing all our results.

\end{abstract}

\section{Introduction}













A same cell environment may lead to different cell-fate decisions. In most cases, it is considered that the phenotype adopted by a cell is determined by the concentration of several molecules in interaction \cite{guantes2008multistable}.  It is now well documented that such `biological switches' can be accurately modeled by multistable ordinary differential equations (ODEs), where each stable state represents a possible phenotype \cite{Thomas2002laws}. 

These models have  been widely used in order to describe different cellular processes such as the epithelial-mesenchymal transition (EMT) \cite{bracken2008double, johnston2005micrornas, siemens2011mir}, hematopoietic stem cells \cite{huang2007bifurcation, laslo2006multilineage, roeder2006towards},  embryonic stem cells \cite{chickarmane2006transcriptional} or other cell-fate differentiation phenomena involved in \emph{Xenopus}  \cite{ferrell1998biochemical, novak1993numerical}, \emph{Drosophila} \cite{papatsenko2011drosophila}  or \emph{Escherichia coli} \cite{gardner2000construction, li2020transition, ozbudak2004multistability}. 

The development of a relevant ODE model hence benefits from a priori knowledge of the possible number of stable states, and how this number evolves in the parameter space. As an example, the epithelial-mesenchymal transition phenomenon involves three different phenotypes, and it is thus crucial to be able to determine minimal conditions allowing the system to be tristable~\cite{lu2013microrna}.  

A general theoretical answer to finding the number of stable states is certainly out of reach for high-dimensional ODEs with a large number of parameters. 
Understanding the more simple building blocks of these complex models, however, remains of paramount importance, even more so with the advent of synthetically-built switches where, to some extent, the model may be chosen and kept simple~\cite{gardner2000construction}.

A widely used starting ODE model writes as follows
\begin{align}
\begin{cases}
\dot x =\alpha f(y)-x\\
\dot y=\beta g(x)-y,
\end{cases}
\quad f', g'<0 \ \mathrm{or} \ f', g'>0
\label{system}
\end{align}
where $x$ and $y$ stand for the (normalized) concentrations of the two molecules ($A$ and $B$ on Figure~\ref{Competitive/Cooperative schemes}), $\alpha>0$ and $\beta>0$ their synthesis rates. Here, $f$ and $g$ are two monotonic functions which model the interactions between these two molecules, and are both strictly increasing or strictly decreasing,  depending on whether the system is cooperative or competitive. A classical choice for $f$ and $g$ are Hill functions, \textit{i.e.}, functions of the form
\[x\mapsto \frac{1}{1+x^r},\]
with $r \geq 1$.
More generally, molecule interactions are usually considered  to behave sigmoidally~\cite{Thomas2002laws, thomas1990biological}.

\begin{figure}[h]
\centering
\includegraphics[width=0.6\textwidth]{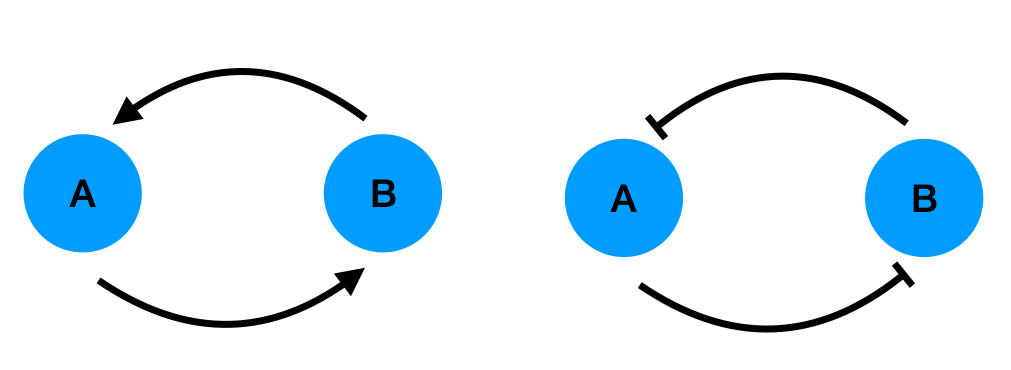}
\caption{Schematic representations of system \eqref{system}: the arrows represent activation ($f$ and $g$ increasing) and the bar inhibition ($f$ and $g$ decreasing). Under the hypothesis of Theorem \ref{theo intro}, we prove that such systems are either monostable or bistable. When bistability holds, the system on the left (cooperative system) has two stable points corresponding to (high A/ high B) and (low A/ low B), while the one on the right (cooperative system) leads to (high A/ low B) and (low A/ high B) stable equilibria.}
\label{Competitive/Cooperative schemes}
\end{figure}
\paragraph{State of the art.}
The seminal paper of Cherry and Adler~\cite{cherry2000make} is the main breakthrough towards understanding  when multistability occurs for such models. Under the condition
\[\underset{y>0}{\sup}\biggl(\left\lvert\frac{yf'(y)}{f(y)}\right\rvert \biggr) \; \underset{x>0}{\sup}\biggl(\left\lvert\frac{xg'(x)}{g(x)}\right\rvert \biggr)>1,\]
it is proven that there exist parameters $\alpha$ and $\beta$ such that system \eqref{system} is multistable. When applied to Hill functions, this shows that multistability will occur for some parameters $\alpha$, $\beta$ whenever $r \in \mathbb{N},\, r \geq 2$ for $f$ or $g$.
Interestingly, the authors noted that the sigmoid shape of $f$ and $g$ is not a necessary condition for bistability.


The numerical investigation of systems such as~\eqref{system} suggests that they are in fact always either monostable or bistable. 
This has led some authors to claim that self-regulation is required in order to get a tristable ODE, \textit{i.e.}, at least one of the cells must have a positive feedback on itself \cite{jia2017operating, gardner2000construction, macia2009cellular}. 

Up to our knowledge, this conjecture of at most bistability is yet to be proven. Moreover, for given functions $f$  and $g$, determining the exact set of parameters $(\alpha, \beta)$ for which this system is monostable or bistable remains difficult without the help of numerical simulations. 
Let us mention the very recent paper~\cite{li2020transition} which, by means of direct computations, solves the specific case of $f : x\mapsto \frac{1}{1+x}$ and $g : x\mapsto \frac{1}{1+x^n}$.


In the present work, we therefore address the following two key questions.
\begin{itemize}
\item Under which conditions is system \eqref{system} at most bistable? 
\item For given functions $f$ and $g$, which parameters $\alpha$, $\beta$ lead to monostability, and which ones lead to bistability?
\end{itemize}

A natural way to answer the first question is to note that a point $(\bar x, \bar y)\in \mathbb{R}_+$ is an equilibrium of~\eqref{system} if and only if 
\begin{align*}
\begin{cases}
\alpha f(\beta g(\bar x))=\bar x\\
\bar y = \beta g(\bar x)
\end{cases}.
\end{align*}
Hence, studying the equilibria of \eqref{system} is equivalent to studying the fixed points of $x\mapsto \alpha f(\beta g(x))$. The main difficulty lies in the fact that, even if $f$ and $g$ are two `simple' functions, there is no reason for $x\mapsto \alpha f(\beta g(x))$ to be as well.  As an example, with Hill functions of integer orders $n$ and $m$,  determining the fixed points of $x\mapsto \alpha f(\beta g(x))$ is equivalent to investigating the positive roots of a polynomial of degree $nm+1$, which proves to be difficult as soon as $nm\geq3$. 

\paragraph{Main results.}
Working around this difficulty, our main result is a simple and general result ensuring at most bistability.
\begin{theorem}
If the functions
\[\frac{1}{ \sqrt{\lvert f' \rvert }} \quad \text{and} \quad \frac{1}{\sqrt{\lvert g' \rvert}}\]
are strictly convex, then, for any $\alpha, \beta >0$, system \eqref{system} has at most three equilibria, among which at most two stable equilibria. 
\label{theo intro}
\end{theorem}
Not only does this result apply to all classically-used functions we are aware of (including Hill and shifted Hill functions), it may easily be checked visually for more involved functions. We show how this result extends to specific cyclic $n$-dimensional ODEs, with the same hypotheses on the functions modeling the mutual regulations. 


Following an approach reminiscent of that of \cite{cherry2000make}, we go further and develop a general method for the identification of which parameters $\alpha$, $\beta$ lead to either monostability or bistability. For this purpose, we develop a framework yielding a condition under which system \eqref{system} has at most or at least $n$ equilibria, for all $n\in \mathbb{N}$.
The obtained criterion is not completely explicit, but the resulting formula  makes it numerically straightforward to check if some chosen parameters induce a monostable or a bistable system. We hence bypass any computationally-expensive grid-search through the parameter space. 
We show that the method in \cite{cherry2000make} corresponds to the case $n=1$ of this general framework.
With the same framework, Theorem~\ref{theo intro} corresponds to studying the case $n=2$. For higher values of $n$, however, we have not been able to apply this theoretical framework as the resulting computations prove to be too intricate. 

We also prove that when bistability occurs, the separatrix between the two basins of attraction is a one-dimensional curve. Note that our proof is implicit and does not provide a formula for the curve, unless some specific symmetry assumptions are made.


Taken together, our results show that system~\eqref{system} will generically lead to either one of the pictures of Figure~\ref{phase}, \textit{i.e.}, we fall into one of these two cases:
\begin{itemize}
\item system \eqref{system} is monostable, and all the solutions converge to the unique equilibrium,
\item system \eqref{system} is bistable,  and, in this case, the basins of attraction of the two stable points are separated by a one-dimensional separatrix which contains the unique other (unstable saddle) equilibrium point. 
\end{itemize}

\begin{figure}[H]
\begin{subfigure}[h!]{0.5\textwidth}
\includegraphics[width=\textwidth]{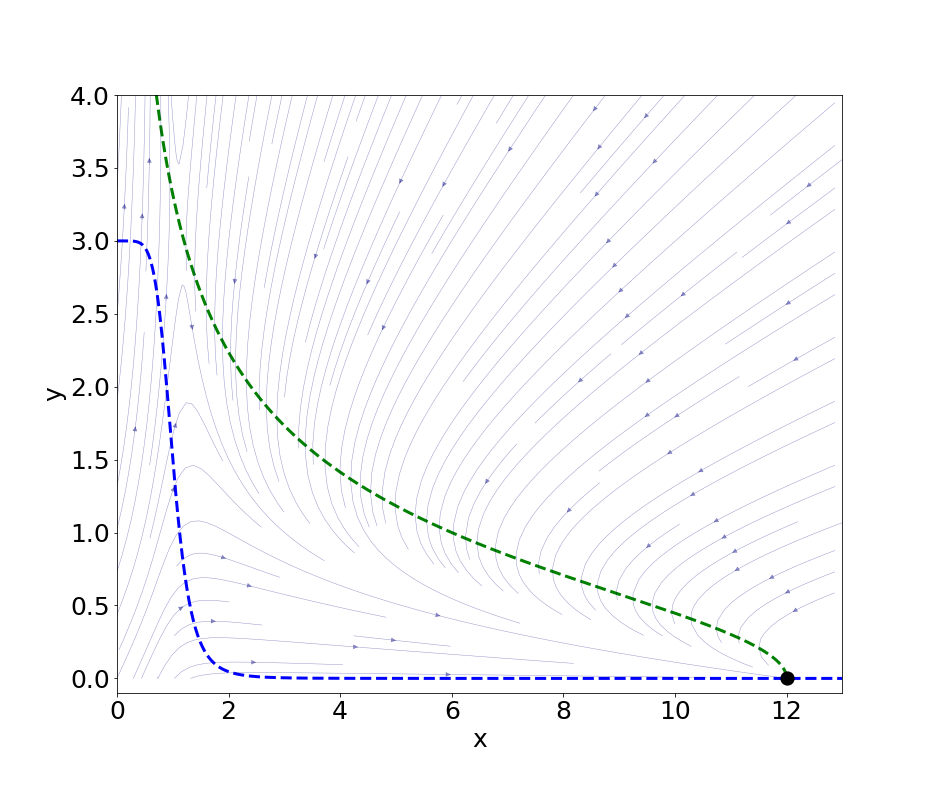}
\end{subfigure}
\begin{subfigure}[h!]{0.5\textwidth}
\includegraphics[width=\textwidth]{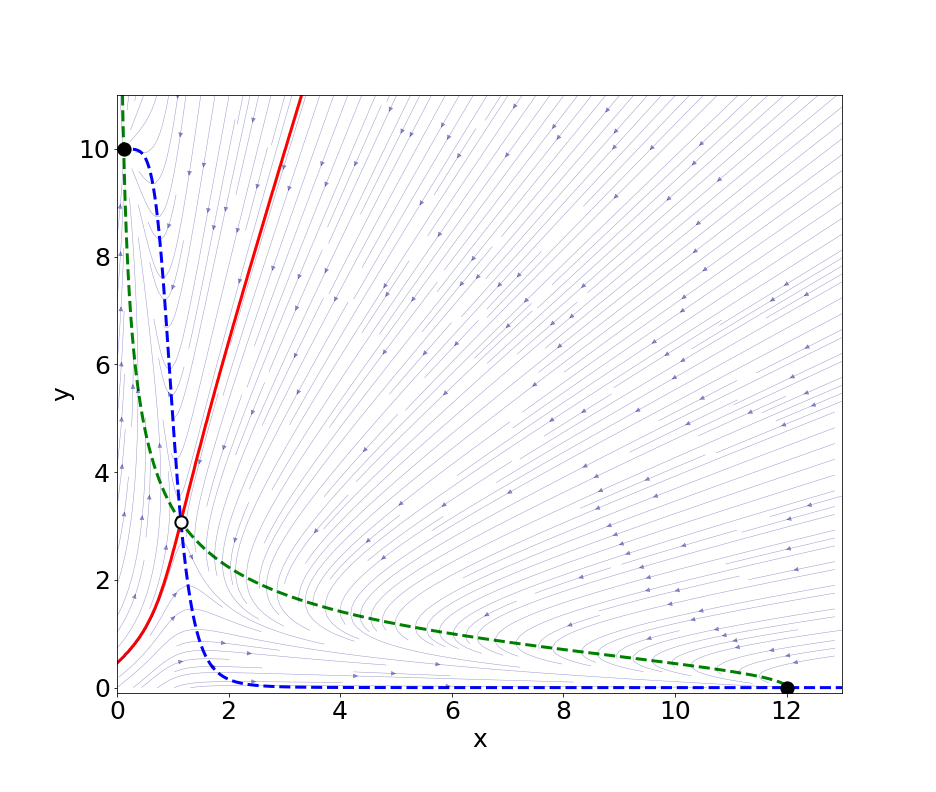}
\end{subfigure}
\caption{Typical phase planes for system~\eqref{system}.
In both simulations, we have taken $f:x\mapsto \frac{1}{1+x^2}$ and $g:x\mapsto \frac{1}{1+x^6}$, and $\alpha=10$. The value of $\beta$ is $3$ in the left diagram, and $12$ in the right one.}
\label{phase}
\end{figure}


\paragraph{Outline of the paper.} The paper is organized as follows. Section~\ref{Section2} is devoted to setting the mathematical framework and some general results which prove to be useful throughout. Next, we prove our main Theorem~\ref{theo intro}, in Section~\ref{Section3}, which involves defining an appropriate class of functions and studying it in detail.
We then turn our attention to finding the parameters  for which the system of interest is either monostable or bistable in Section~\ref{Section4}.
We present a generalizing framework, and then apply it to the cases of mono and bi-stability. We also compute the parameters for the cases of interest found in the literature, such as the toggle switch of~\cite{gardner2000construction}. 




\section{Preliminary results}
\label{Section2}
Throughout this article, we study systems of the form
\begin{equation}
\begin{cases}
\dot x = \alpha f(y)-x\\
\dot y=\beta g(x)-y,
\end{cases}
\label{initial system}
\end{equation}
starting from  an initial condition $(x_0, y_0)$ with $x_0 \geq 0, y_0\geq 0$.
Here
\begin{itemize}
\item $\alpha, \beta$ are two positive parameters,
\item  $f, g \in C^1(\mathbb{R}_+, \mathbb{R}_+)$ are two increasing or two decreasing functions, and at least one of these functions is bounded. 
\end{itemize}

If $f$ and $g$ are increasing, system \eqref{initial system} is called \emph{cooperative}. Examples of cooperative systems include two-species interactions that benefit both species, a kind of interaction often referred to as mutualism \cite{Vargas2012}.
On the other hand, if $f$ and $g$ are decreasing, system \eqref{initial system} is called \emph{competitive}. A simple example of a competitive system is the well-studied ``genetic switch'' of two proteins that each repress the synthesis of the other \cite{cherry2000make,jia2017operating}.

Without loss of generality, one can assume that $f,g>1$ on $\mathbb{R}_+$ (in the case where $f$ and $g$ are increasing) or $f,g<1$ on $\mathbb{R}_+$ (in the case where $f$ and $g$ are decreasing).
Under these conditions, since $f$ or $g$ is bounded, all the solutions of this ODE are bounded, regardless of its initial condition. It is well-known that any solution of such a system converges to an equilibrium point \cite{hirsch1982systems}.

Thus, the analysis of system \eqref{initial system} requires studying its equilibrium points. In this section, we begin by providing important results on how the equilibrium points are ordered, as well as on their basins of attraction.
These will be the starting point of our investigation of the system's multistability.

\subsection{Ordering of stable points}

Since the functions $f$ and $g$ are one-to-one, it is clear that  $(\bar x, \bar y)\in \mathbb{R}_+^2$ is an equilibrium point of \eqref{initial system} if and only if 
\begin{align*}
\begin{cases}
\alpha f(\beta g(\bar x))=\bar x\\
\bar y =\beta g(\bar x)
\end{cases}.
\end{align*}
Hence, the number of equilibria of \eqref{initial system} is equal to the number of fixed points of 
\[F : x\longmapsto \alpha f(\beta g(x)).\] Moreover, if \eqref{initial system} has a finite number of equilibria, that we denote $(x_1, y_1)$, $(x_2, y_2)$, ... $(x_n, y_n)$, with $x_1<x_2<...<x_n$, then
\begin{itemize}
\item $y_1<y_2<...<y_n$ if system \eqref{initial system} is cooperative (\textit{i.e.} if $f$ and $g$ are increasing);
\item $y_1>y_2>...>y_n$ if system \eqref{initial system} is competitive (\textit{i.e.} if $f$ and $g$ are decreasing). 
\end{itemize}

Let us give the single notion of stability of an equilibrium point that we shall make use of throughout.
\begin{definition}
We say that an equilibrium point $(\bar x, \bar y$) is asymptotically stable if there exists a neighborhood $U$ of  $(\bar x, \bar y$) such that if $(x_0, y_0) \in U$, the trajectory $(x(t), y(t))$ starting from $(x_0, y_0)$ satisfies \[\lim_{t \rightarrow+\infty} (x(t),y(t)) = (\bar x, \bar y).\]
\end{definition}
\noindent
We will abusively refer to \textit{stable equilibrium points} when dealing with asymptotic stable points.

From the applicative point of view, in particular, the ordering of equilibrium points means the following. If \eqref{initial system} is bistable, the two stable points will be of the type (low $x$/ low $y$) and (high $x$/ high $y$) if the system is cooperative, and of the type (low $x$/ high $y$) and (high $x$/ low $y$) if it is competitive (see Figure \ref{phase}). 

We also observe that for any $(x,y)\in \mathbb{R}_+$, the Jacobian matrix of the right-hand side of this ODE in $(x,y)$ is 
\begin{align*}
J_{(x,y)}=
\begin{pmatrix}
-1 & \alpha f'(y)\\
\beta g'(x) & -1
\end{pmatrix}.
\end{align*}
Since $\mathrm{Tr}(J_{(x,y)})=-2$ and $\mathrm{det}(J_{(x,y)})=1-\alpha \beta f'(y)g'(x)$, a fixed point $(\bar x, \bar y)$ is stable if $\alpha \beta f'(\bar y)g'(\bar x)<1$ and unstable if $\alpha \beta f'(\bar y)g'(\bar x)>1$. In other words, under the hypothesis ``for any fixed point $(\bar x, \bar y)$,  $\alpha\beta f'(\bar y)g'(\bar x)\neq 1$'', the number of stable equilibria of \eqref{initial system} is equal to the number of time $F$ crosses the identity line `from above', and the number of unstable equilibria to the number of times $F$ crosses the identity line `from below'. Therefore, since $F$ is positive, increasing and bounded, system $\eqref{initial system}$ has $d$ stable equilibria if and only if it has $d-1$ unstable equilibria. This result is proven rigorously in the next section. 


\subsection{Basins of attraction}

Due to the particular shape of the system that we study, we have a precise result regarding the basins of attraction: if the system is monostable, then all solutions converge to the stable point, meaning that it is globally asymptotically stable. If it is bistable, then the two basins of attractions are separated by a separatrix, which is the curve of an increasing function if \eqref{initial system} is competitive and of a decreasing function if \eqref{initial system} is cooperative (see Figure \ref{phase}).  
This result relies on the following proposition.
\begin{prop}
Let us consider the ODE system 
\begin{align}
\begin{cases}
\dot x = F(y)-x\\
\dot y =G(x)-y
\end{cases},
\label{system without parameters}
\end{align}
where $F$ and $G$ are either both increasing or both decreasing, and at least one of them is bounded.  
We assume that there exists $(\bar x_u$, $\bar y_u) \in \mathbb{R}_+^2$ an equilibrium point of \eqref{system without parameters} such that 
\[G'(\bar x_u)F'(\bar y_u)>1\]
Then, the basin of attraction of $(\bar x_u$, $\bar y_u)$ has measure zero in $\mathbb{R}_+^2$. More precisely, this basin of attraction is included in a curve of the shape
\[\{(x, \gamma(x)), x\in (a, b)\},\]
where $a\geq 0$,  $b\in \mathbb{R}_+\cup\{+\infty\}$ and $\gamma : (a, b) \rightarrow \mathbb{R}_+$ is a continuous function. This function is decreasing if $F$ and $G$ are increasing, and increasing if $F$ and $G$ are decreasing. 

\label{separatrix}

\end{prop}

\begin{proof}
We prove this result only in the competitive case: the arguments can easily be adapted to the cooperative case. 
 First, let us note that \eqref{system without parameters} is strictly competitive, in the sense that when we rewrite \eqref{system without parameters} as $\dot z = \Gamma(z)$,
 the vector field $\Gamma=(\Gamma_1, \Gamma_2)$ underlying the ODE satisfies for all $z\in \mathbb{R}^2$,
\begin{align*}
\frac{\partial \Gamma_1}{\partial z_2}(z)<0, \quad
\frac{\partial \Gamma_2}{\partial z_1 }(z)<0.
\end{align*}
The strict competitiveness of the system implies that it satisfies the comparison principle, which writes as follows.
Let $z$, $w$ be two solutions such that 
\begin{align*}
\begin{cases}
z_1(0)\leq w_1(0) & \quad (\text{resp. } z_1(0)\geq w_1(0))\\ 
z_2(0)>w_2(0)  & \quad (\text{resp. } z_2(0)<w_2(0)).
\end{cases}
\end{align*}
Then, for any $t>0$ such that $z$ and $w$ are defined on $[0, t]$, 
\begin{align*}
\begin{cases}
z_1(t)<w_1(t) \quad (\text{resp. } z_1(t)> w_1(t))\\
z_2(t)>w_2(t) \quad (\text{resp. } z_1(t)< w_1(t)).
\end{cases}
\end{align*}
Let $(x_u, y_u)$ be a solution of system \eqref{system without parameters} which converges to some equilibrium point $(\bar x_u, \bar y_u)$, and $(x,y)$ a solution of \eqref{system without parameters} such that
\begin{align*}
\begin{cases}
x(0)<x_u(0)\\
y(0)\geq y_u(0).
\end{cases}
\end{align*}
We assume that $(x,y)$ converges to $(\bar x_u, \bar y_u)$.
An application of the comparison principle entails
\begin{align*}
\forall t >0, \qquad w_1(t):=x_u(t)-x(t)>0, \quad
w_2(t):=y(t)-y_u(t)>0.
\end{align*}
Hence, we may write
\begin{align*}
\dot w_1&=\dot x_u-\dot x =F(y_u)-x_u-F(y)+x
=\frac{F(y_u)-F(y)}{y-y_u}w_2-w_1.
\end{align*}
Likewise,
\[\dot w_2 = \frac{G(x)-G(x_u)}{x_u-x}w_1-w_2.\]
Since $(x, y)$ and $(x_u, y_u)$ both converge to $(\bar x_u, \bar y_u)$, and $F'(\bar y_u)G'(\bar x_u)>1$, there exist $c, d>0$ which satisfy $c d >1$, and $T>0$  such that for all $t\geq T$ : 
\[\frac{F(y_u(t))-F(y(t))}{y(t)-y_u(t)}>c \quad \text{and} \quad \frac{G(x(t))-G(x_u(t))}{x_u(t)-x(t)}>d.\]
Therefore, for all $t\geq T$,
$\dot w_1(t) \geq c\, w_2(t)-w_1(t) $ and $\dot w_2(t)  \geq d\, w_1(t)-w_2(t)$.
We now consider 
\[W:=(1+d)w_1+(1+c)w_2.\]
According to the previous computations, we have,  for all $t\geq T$, 
\begin{align*}
 \dot W(t) \geq & -(1+d)w_1(t)+c(1+d) w_2(t) +d (1+c) w_1(t)-(1+c)w_2(t)\\
 =&\underbrace{(cd-1)}_{>0}\underbrace{(w_1(t)+w_2(t))}_{>0}>0.
\end{align*}
Hence, $W$ does not converge to zero, which contradicts the fact that $w_1$ and $w_2$ converge to zero. 
With the same reasoning, we prove that if $x(0)>x_u(0)$ and $y(0)\leq y_u(0)$,  then $(x,y)$ does not converge to $(\bar x_u, \bar y_u)$. 

In particular, for all $x_0 \in \mathbb{R}_+$, there exists at most one $y_0\in \mathbb{R}_+$ such that $(x_0, y_0)$ is in the basin of attraction of $(\bar x_u, \bar y_u)$, which proves the existence of $\gamma$.  Moreover, if $(x_0,y_0)$ and $(x_0',y_0')$ are two points of this basin of attraction such that $x_0<x_0'$, then $y_0<y_0'$, which shows that $\gamma$ is increasing. The continuity of $\gamma$ and the connectedness of the set follow from the continuity of the solutions. 

\end{proof}

\section{A criterion which ensures at most bi-stability}
\label{Section3}

The main purpose of this section is to prove the following theorem: 

\begin{theorem}
Let $f, g \in C^3(\mathbb{R}_+, \mathbb{R}_+)$ two functions such that $f'>0, g'>0$ or $f'<0, g'<0$ on $\mathbb{R}_+$. 
We recall that system $\eqref{initial system}$ refers to 
\begin{align*}
\begin{cases}
\dot x=\alpha f(y)-x\\
\dot y =\beta g(x)-y
\end{cases}.
\tag{\ref{initial system}}
\end{align*} 
Then
\begin{itemize}
\item If $\frac{1}{\sqrt{\lvert f'\rvert }}$ and $\frac{1}{\sqrt{\lvert g'\rvert }}$ are convex (we say that $f$ and $g$ are $\gamma^{1/2}$-convex), and at least one of these functions is strictly convex, then for any $\alpha, \beta >0$ system \eqref{initial system} has at most three equilibria, and is either monostable or bistable. 
\item If $\frac{1}{\sqrt{\lvert f'\rvert }}$ and $\frac{1}{\sqrt{\lvert g'\rvert }}$ are concave (we say that $f$ and $g$ are $\gamma^{1/2}$-concave),  and at least one of these functions is strictly concave, then for any $\alpha, \beta >0$, system \eqref{initial system} is monostable. 
\end{itemize}
\label{theorem atmostbistable}
\end{theorem}

In order to determine if a given function is $\gamma^{1/2}$-convex or $\gamma^{1/2}$-concave, we provide different properties about these functions, which are summarized below.
\begin{itemize}
\item A function $f$ is $\gamma^{1/2}$-convex (resp. $\gamma^{1/2}$-concave) if and only if $f'f^{(3)}\leq {f''}^2$ (resp. $f'f^{(3)}\geq {f''}^2$).
\item If $f$ and $g$ are $\gamma^{1/2}$-convex (resp. $\gamma^{1/2}$-concave), then $f\circ g$ is  $\gamma^{1/2}$-convex (resp. $\gamma^{1/2}$-concave).
\item $x\mapsto x^a$ is strictly  $\gamma^{1/2}$-convex if $\lvert a\rvert >1$,  strictly $\gamma^{1/2}$-concave if $0<\lvert a \rvert<1$. 
\item  The only functions which are both $\gamma^{1/2}$-convex and $\gamma^{1/2}$-concave are the affine and the homographic functions. 
\item A strictly monotonic function $f$ is $\gamma^{1/2}$-convex if and only if $f^{-1}$ is $\gamma^{1/2}$-concave. 
\end{itemize}
The rest of this section is devoted to proving Theorem~\ref{theorem atmostbistable} and the above properties regarding $\gamma^{1/2}$-convexity.

\subsection{A priori bounds on the number of fixed points}

In what follows,  $I$ will denote an arbitrary (possibly unbounded) interval of $\mathbb{R}$. 

The first proposition and its corollary prove an intuitive fact about the fixed points of a function and the sign of its derivative at this point. As explained in the preliminary results, this proposition is the basis for all the results of this section. 
\begin{prop}
Let $f\in C^1(I, \mathbb{R})$. If $f$ has $2n$ fixed points or more ($n\in \mathbb{N}^*$), then there exist  $x^-_1<x^-_2<\ldots<x^-_n$ and $x^+_1<x^+_2<\ldots<x^+_n$ some fixed points of $f$ such that, for any $i\in \{1,\ldots,n\}$:
\[f'(x_i^-)\leq1 \quad \text{and} \quad f'(x_i^+)\geq 1.\]
Conversely, if there exist $n$ fixed points of $f$, denoted $x_1^+<x_2^+<\ldots<x_n^+$ such that 
\[\forall i \in \{1,\ldots, n\}, \quad  f'(x_i)<1 \quad \text{or } \quad \forall i \in \{1,\ldots, n\}, \quad f'(x_i)>1,\] 
then $f$ has at least $2n-1$ fixed points. 
\label{prop fixed_point_derivative}
\end{prop}

\begin{proof}
The first implication can clearly be proven by induction. The main difficulty lies in the base case, \textit{i.e.} in the case $n=1$. 
Let us denote 
\[F:=f-\mathrm{id}.\]
If $F$ has a finite number of roots, or more generally, if the set of the roots of $F$ does not have an accumulation point, the result immediately  holds, since $F$ reaches it roots `from above' and `from below' alternatively.  Otherwise, we consider a bounded sequence of roots of $F$, denoted $(c_n)$, and we assume that for any $n \in \mathbb{N}$, $F'(c_n)<0$ (or $F'(c_n)>0$). Since $(c_n)$ is bounded, we can extract a convergent subsequence, and we denote $c$ its limit. According to the continuity of $F$, $F(c)=0$ and, according to its differentiability, $F'(c)=\underset{n\rightarrow +\infty}{\lim}\frac{F(c)-F(c_n)}{c-c_n}=0$, which proves the result. 

The converse implication simply stems from applying the intermediate value theorem to the function $F$.

\end{proof}

\begin{cor}
Let us assume that $I=\mathbb{R}$ (or $\mathbb{R}_+^*$) and that $f$ is positive and bounded.
If $f$ has $2n$ fixed points or more ($n\in \mathbb{N}^*$), then there exist $x_1^-<\ldots <x_{n+1}^-$ some fixed points of $f$ such that for any $i\in \{1, \ldots, n+1\}$
\[f'(x_i^-)\leq 1.\]
Conversely, if there exist $n$ fixed points of $f$ (denoted $x_1, \ldots, x_n$) such that 
\[\forall i\in \{1, \ldots, n\}, \quad f'(x_i)>1,\]
then $f$ has at least $2n+1$ fixed points. 
\label{cor fixed points f positive}
\end{cor}

\begin{proof}
We note that 
\[E:=\{x\in \mathbb{R} : f(x)=x\}\]
is a closed and bounded set, and thus that it is compact. Hence, it has a minimum and a maximum element, that we denote $x^m$ and $x^M$. 
Since for all $x<x^m$ and all $x>x^M$, $f(x)>x$, then $f'(x^m)\leq 1$, $f'(x^M)\leq 1$. 

Therefore, if $f$ has at least $2n$ fixed points, then it has at least $2n-2$ fixed points on $(x^m, x^M)$. We conclude by applying the previous proposition on $(x^m, x^M)$.
\end{proof}

\subsection{Properties of $\gamma^{1/2}$-convex functions}

We now use these two results (in the specific cases $n=1$ and $n=2$) in order to establish a convexity criterion related to the derivative of the function of interest. It ensures that this function cannot have more than three fixed points.
\begin{prop}
Let $f\in C^1(I, \mathbb{R})$. Let us assume that there exists $\Psi : f'(I)\rightarrow \mathbb{R}$ a function which satisfies the four following conditions:
\begin{enumerate}[(i)]
\item $\Psi(1)=1$
\item $\forall x\leq 1, \Psi(x)\geq 1$ 
\item $\forall x\geq 1, \Psi(x)\leq 1$ 
\item $\Psi \circ f'$ is strictly convex or strictly concave. 
\end{enumerate}
Then, $f$ has at most three fixed points.  In the case where $\Psi \circ f'$ is strictly concave, if we make the stronger assumption that $I=\mathbb{R} \  \mathrm{ or } \  \mathbb{R}_+$ and $f$ is bounded, then $f$ has a unique fixed point.  

In particular, if $f'>0$ and if there exists $\alpha>0$ such that 
\[\gamma_f^\alpha:=\biggl(\frac{1}{f'}\biggr)^\alpha\]
is strictly convex or strictly concave, then $f$ has at most three fixed points. 
\label{Psi functions}
\end{prop}

\begin{proof}
We show the contrapositive. 
Let us assume that $f$ has four fixed points or more, and let $\Psi$ be a function which satisfies the first three points of the proposition. 
According to Proposition \ref{prop fixed_point_derivative}, there exist $x_1, x_2, y_1, y_2$ four distinct points such that, for $i\in \{1, 2\}$
\begin{align*}
f(x_i)&=x_i \quad \text{and} \quad f(y_i)=y_i\\
f'(x_i)&\geq 1 \quad \text{and} \quad  f'(y_i)\leq 1.
\end{align*}
According to the mean value therorem, there exists $\theta, \theta' \in (0,1)$ such that 
\[f'(\theta x_1 +(1-\theta)x_2)=\frac{f(x_1)-f(x_2)}{x_1-x_2}=1 \quad \text{and} \quad f'(\theta' y_1 +(1-\theta')y_2)=\frac{f(y_1)-f(y_2)}{y_1-y_2}=1.\]
Thus,
\[\theta \Psi(f'(x_1))+(1-\theta)\Psi(f'(x_2))\leq \theta +(1-\theta )=1=\Psi(f'(\theta x_1+(1-\theta )x_2))\]
and 
\[\theta' \Psi(f'(y_1))+(1-\theta')\Psi(f'(y_2))\geq \theta' +(1-\theta' )=1=\Psi(f'(\theta' y_1+(1-\theta' )y_2)),\]
which proves that $\Psi \circ f$ is neither strictly convex nor strictly concave.  

When $f$ is bounded and positive, we argue similarly but with the help of Corollary~\ref{cor fixed points f positive}. 
\end{proof}
We recall and extend the definition of $\gamma_f^\alpha$ for any $\alpha>0$ and function $f\in C^1(I, \mathbb{R})$ such that $\lvert f' \rvert >0:$
\begin{align*}
\gamma_f^\alpha : I &\longrightarrow \mathbb{R}_+\\
x & \longmapsto \biggl(\frac{1}{\lvert f'(x)\rvert}\biggr)^\alpha.
\end{align*}
The following lemma provides a straightforward way to decide whether a given function $f$ satisfies one of the `$\gamma^\alpha$-properties'. 
\begin{lem}
Let $f\in C^3(I, \mathbb{R})$  a strictly monotonic function.
Then $\gamma_f^\alpha$ is convex (resp. concave) if and only if 
\[f'f^{(3)}\leq(\alpha+1){f''}^2 \quad (\text{resp. }   f'f^{(3)}\geq (\alpha+1){f''}^2)\]
on $I$,  and strictly convex (resp. strictly concave) if and only if 
\[f'f^{(3)}<(\alpha+1){f''}^2 \quad (\text{resp. } f'f^{(3)}> (\alpha+1){f''}^2) \]
on a dense subset of $I$. 
\label{lemma equivalence convexity}
\end{lem}

\begin{proof}
First, let us note that for any $f$ which satisfies the hypotheses, we have $\gamma_{-f}^\alpha=\gamma_{f}^\alpha$. We may hence assume $f'>0$ without loss of generality. 

We recall that a function $F\in C^2(I, \mathbb{R})$ is strictly convex (resp. strictly concave) if and only if $F''>0$ (resp. $F''<0$) on a dense subset of $I$. Computing the second derivative of $\gamma_f^\alpha$, we find 
\[{\gamma_f^\alpha}''=-\alpha\frac{f'f^{(3)}-(\alpha+1){f''}^2}{{f'}^{\alpha+2}}.\]
The result immediately follows. 
\end{proof}

In general, if $\gamma_f^\alpha$ and $\gamma_g^\alpha$ are convex (or concave), $\gamma_{f\circ g}^\alpha$ has no reason to be. The case $\alpha=\frac{1}{2}$ stands out since the property is stable under composition, as we shall see. Since $\gamma_{f}^{1/2}$ will play a crucial role in what follows, we recall its definition here:
\begin{definition}[$\gamma^{1/2}$-convex function]
Let $f\in C^3(I, \mathbb{R})$ be a strictly monotonic function.
We recall that 
\begin{align*}
\gamma_f^{1/2}:I &\longrightarrow \mathbb{R}\\
x&\longmapsto\frac{1}{\sqrt{\lvert f'(x)\rvert}}.
\end{align*} 
We say that 
\begin{itemize}
\item $f$ is $\gamma^{1/2}$-convex if $\gamma_f^{1/2}$ is convex ($\iff$ $f'f^{(3)}\leq \frac{3}{2}{f''}^2$ on $I$).
\item $f$ is strictly $\gamma^{1/2}$-convex if $\gamma_f^{1/2}$ is strictly convex ($\iff$ $f'f^{(3)}<\frac{3}{2}{f''}^2$ on a dense subset of~$I$).
\item $f$ is $\gamma^{1/2}$-concave if $\gamma_f^{1/2}$ is concave ($\iff$ $f'f^{(3)}\geq \frac{3}{2}{f''}^2$ on $I$).
\item $f$ is srictly $\gamma^{1/2}$-concave if $\gamma_f^{1/2}$ is strictly concave ($\iff$ $f'f^{(3)}> \frac{3}{2}{f''}^2$ on a dense subset of~$I$).
\end{itemize} 
\end{definition}

\begin{prop}
Let $g\in C^3(I, \mathbb{R})$,  $f\in C^3(g(I), \mathbb{R})$. Then
\[ f, g \;\;  \gamma^{1/2}-\text{convex (resp. $\gamma^{1/2}$-concave)} \quad \implies \quad f\circ g\text{ $\gamma^{1/2}$-convex (resp. $\gamma^{1/2}$-concave)}.\]  Furthermore, if in the above either $f$ or $g$ is strictly $\gamma^{1/2}$-convex (resp. $\gamma^{1/2}$-concave), then $f\circ g$ is strictly $\gamma^{1/2}$-convex (resp. $\gamma^{1/2}$-concave). 
\label{compose}
\end{prop}

\begin{proof}
Let us denote $h:=f \circ g$ and compute its third-order derivative:
\begin{align*}
h^{(3)}&=g^{(3)}(f\circ g)+3g'g''(f'\circ g) +{g'}^3 (f'' \circ g).
\end{align*}
Thus
\begin{align*}
h'h^{(3)}-(\alpha+1){h''}^2=&{g'}^4\bigl[(f'\circ g) (f^{(3)}\circ g)-(\alpha+1)(f''\circ g)^2\bigr]\\
&+(f'\circ g)^2\bigl[g'g^{(3)}-(\alpha+1){g''}^2\bigr]
+(3-2(\alpha+1)){g'}^2g''(f'\circ g)(f''\circ g).
\end{align*}
In the particular case where $\alpha=\frac{1}{2}$, $3-2(\alpha+1)=0$, which shows that the last term vanishes.
The previous lemma concludes the proof since the other two terms have the appropriate signs.
\end{proof}

We now give examples of $\gamma^{1/2}$-convex/concave functions. The main  point is that a function $f$ is $\gamma^{1/2}$-convex/concave if and only if, for any affine or homographic function $h$, the composite functions $h\circ f$ and $f \circ g$ are also $\gamma^{1/2}$-convex/concave.
\begin{ex}
\begin{enumerate}
\item $f$ is both $\gamma^{1/2}$-convex  and $\gamma^{1/2}$-concave if and only if $f$ is an affine function or a homographic function, \textit{i.e.}, if there exist $a,b,c,d \in \mathbb{R}$ such that 
\[f = x\longmapsto \frac{ax+b}{cx+d}.\]
\item We consider, for all $\nu\in \mathbb{R}$, the function $x\mapsto x^\nu$ on $\mathbb{R}_+^*$.
\begin{itemize}
\item If $\lvert \nu \rvert >1$, then $x\mapsto x^\nu$ is strictly $\gamma^{1/2}$-convex on $\mathbb{R}_+^*$.
\item If $\lvert \nu \rvert <1$, then $x\mapsto x^\nu$is strictly $\gamma^{1/2}$-concave on $\mathbb{R}_+^*$. 
\end{itemize}
\item $x\mapsto e^x$ is strictly $\gamma^{1/2}$-convex on $\mathbb{R}$.
\end{enumerate}
\end{ex}
\begin{rem}
According to Proposition \ref{compose}, all the functions of the form 
\[x\longmapsto \biggl(\frac{ax^\nu+b}{cx^\nu+d}\biggr)^\mu,\]
with $\nu, \mu\geq 1$ and $\nu >1$ or $\mu >1$ are strictly  $\gamma^{1/2}$-convex. Many common sigmoid functions are $\gamma^{1/2}$-convex, as shown in Appendix~\ref{AppC}. 
\end{rem}
\begin{proof}
\begin{enumerate}
\item  First, let us assume that $f : x \mapsto \frac{ax+b}{cx+d}$ (which is defined on $\mathbb{R}$ if $c=0$ and on $(-\infty, -d/c)\cup \, (-d/c, +\infty)$ otherwise).

Then, for any $x$ in its domain,
\[f'(x)=\frac{ad-cb}{(cx+d)^2}, \quad f''(x)=-2c\frac{ad-cb}{(cx+d)^3},\quad  f^{(3)}(x)=6c^2\frac{ad-cb}{(cx+d)^4}.\]
Thus, 
\[f'(x)f^{(3)}(x)-\frac{3}{2}f''(x)^2=0.\]
Hence, $f$ is indeed $\gamma^{1/2}$-convex and  $\gamma^{1/2}$-concave. 

Conversely, let us assume that $f$ is both $\gamma^{1/2}$-convex and  $\gamma^{1/2}$-concave, \textit{i.e.}, that $\frac{1}{\sqrt{\lvert f'\rvert}}$ is an affine function. Then, there exist $\alpha, \beta \in \mathbb{R} $ such that for all $x\geq 0$
$\frac{1}{\sqrt{\lvert f'(x)\rvert }}=\alpha x + \beta $. 
Thus, 
\[f'(x)=\pm \frac{1}{(\alpha x + \beta)^2}.\]
Integrating, we conclude that $f$ is of the announced form $x\mapsto \frac{ax+b}{cx+d}$.
\item We recall that the function $x\mapsto x^\beta$ is strictly convex on $\mathbb{R}_+^*$ if and only if $\beta>1$ or $\beta <0$ and strictly concave if and only if $\beta \in (0,1)$. 

Hence, since the derivative of $x\mapsto x^\nu$ is $x\mapsto \nu x^{\nu-1}$, $x\mapsto x^\nu$ is strictly $\gamma^{1/2}$-convex if and only if $\lvert \nu\rvert >1$.
Likewise, $x\mapsto x^\nu$ is strictly $\gamma^{1/2}$- concave if and only if $\lvert \nu\rvert <1$.
\item Since, for all $x\in \mathbb{R}$, 
\[\frac{1}{\sqrt{(e^x)'}}=e^{-\frac{1}{2}x},\]
the results immediately follows, owing to the strict convexity of $x\mapsto e^{-\frac{1}{2}x}$.
\end{enumerate}
\end{proof}

\paragraph{Non-local characterization.}
The $\gamma^{1/2}$-convexity/concavity of a function can surprisingly be written under a non-local form. Although this characterization is generally less suitable to check if a given function is $\gamma^{1/2}$-convex/concave, it can be useful to understand why the $\gamma^{1/2}$-convexity/concavity is especially adapted to our problem. Indeed, the fact that the  `$\gamma^{1/2}$-properties' are stable by composition and that they ensure that a function has at most three fixed points can easily be proven by considering this non-local form. 

\begin{prop}
Let $f\in C^3(I, \mathbb{R})$. Then 
\begin{itemize}
\item $f$ is $\gamma^{1/2}$-convex  (resp. $\gamma^{1/2}$-concave) if and only if for any $x,y\in I$, $x\neq y$,
\[f'(x)f'(y)\leq \biggl(\frac{f(x)-f(y)}{x-y}\biggr)^2 \quad (\text{resp.  } f'(x)f'(y)\geq \biggl(\frac{f(x)-f(y)}{x-y}\biggr)^2).\]
\item $f$ is strictly $\gamma^{1/2}$-convex (resp. strictly $\gamma^{1/2}$-concave)  if and only if for any $x,y\in I$, $x\neq y$,
\[f'(x)f'(y)< \biggl(\frac{f(x)-f(y)}{x-y}\biggr)^2 \quad (\text{resp. } f'(x)f'(y)> \biggl(\frac{f(x)-f(y)}{x-y}\biggr)^2).\]
\end{itemize}
\label{non-local}
\end{prop}

\begin{proof}
We only prove the $\gamma^{1/2}$-convex case.
Let us assume that $f$ is $\gamma^{1/2}$-convex, and let $y\in I$. We define $I_y^-=\{x\in I : x<y\}$ and  $I_y^+=\{x\in I : x>y\}$, and we introduce the homographic function
\[h_y(x):=\frac{f'(y)}{x-f(y)}\]
which is well defined on $(-\infty, f(y))$ and on $\mathopen{(}f(y), +\infty)$. As seen in the examples, and according to Proposition \ref{compose}, $h_y \circ f$ is $\gamma^{1/2}$-convex on $I_y^-$ and on $I_y^+$. 
In other words, $\gamma_{h_y\circ f}^{1/2}=\frac{1}{\sqrt{\lvert (h_y\circ f\rvert)'}}$ is convex on $I_y^-$ and on $I_y^+$, and for all $x\in I\backslash \{y\}$,
\[\gamma_{h_y\circ f}^{1/2}(x)=\frac{\lvert f(x)-f(y) \rvert }{\sqrt{f'(x)f'(y)}}.\]
We observe that $\gamma_{h_y\circ f}^{1/2}$ can be extended by continuity at $y$ by defining $\gamma_{h_y\circ f}^{1/2}(y)=0$.
We now assume that $f'>0$ (the case $f'<0$ is similar).  Then, for any $x \in I \backslash \{y\}$, 
\[{\gamma_{h_y\circ f}^{1/2}}'(x)=\mathrm{sgn}(x-y)\biggl(\sqrt{\frac{f'(x)}{f'(y)}}-(f(x)-f(y))\frac{f''(x)}{2(f'(x)f'(y))^{3/2}}\biggr).\]
Thus, 
\[{\gamma_{h_y\circ f}^{1/2}}'(x)\underset{x\rightarrow y^-}{\longrightarrow} -1 \quad \text{and} \quad {\gamma_{h_y\circ f}^{1/2}}'(x)\underset{x\rightarrow y^+}{\longrightarrow} 1.\]
Since $\gamma_{h_y\circ f}^{1/2}$ is convex on $I_y^-$ and on $I_y^+$, $\gamma_{h_y\circ f}^{1/2}$ lies above its tangent at $y$ on these two intervals, \textit{i.e.},
\begin{itemize}
\item  For all $x\in I_y^+$, 
$\gamma_{h_y\circ f}^{1/2}\geq -(x-y).$
\item For all $x\in I_y^-$, 
$\gamma_{h_y\circ f}^{1/2}\geq (x-y).$
\end{itemize}
Therefore, for all $x\in I\backslash \{y\},$
\[\frac{\lvert f(x)-f(y) \rvert }{\sqrt{f'(x)f'(y)}}\geq \lvert x-y \rvert,\]
which proves
\[f'(x)f'(y)\leq \biggl(\frac{f(x)-f(y)}{x-y}\biggr)^2.\]
If $f$ is strictly $\gamma^{1/2}$-convex, then these inequalities become strict, which proves the second point. 

Now, let us prove the converse assertion. 
We assume that for any $x,y\in I$, $x\neq y$,
\[f'(x)f'(y)\leq \biggl(\frac{f(x)-f(y)}{x-y}\biggr)^2.\]
Then for any $x>0, \varepsilon\in \mathbb{R}$ small enough, 
\[f'(x)f'(x+\varepsilon)-\biggl(\frac{f(x+\varepsilon)-f(x)}{\varepsilon}\biggr)^2\leq 0.\]
A Taylor-expansion of the left-hand side leads to
\[\varepsilon^2\left(\frac{1}{6}f'(x)f^{(3)}(x)-\frac{1}{4}f''(x)^2\right)+o(\varepsilon^2)\leq 0,\]
which implies the inequality
\[f'(x)f^{(3)}(x)\leq \frac{3}{2}f''(x)^2.\]
In the case where, for any $x\neq y$
\begin{equation*}
f'(x)f'(y)< \biggl(\frac{f(x)-f(y)}{x-y}\biggr)^2,
\end{equation*}
this inequality becomes strict in a dense subset of $I$. Indeed, let us assume that there exists a segment $J\subset I$ such that 
\[f'f^{(3)}= \frac{3}{2}f''^2 \quad \text{on } J. \]
Then, according to the examples, $f_{|J}$ is a homographic or an affine function, \textit{i.e.}, there exist $a,b,c,d \in\mathbb{R}$ such that:
\[\forall x\in J,\quad f(x)=\frac{ax+b}{cx+d}.\]
Then, for all $x,y\in J$, 
\[f'(x)f'(y)=\biggl(\frac{(ad-cb)}{(cx+d)(cy+d)}\biggr)^2=\biggl(\frac{f(x)-f(y)}{x-y}\biggr)^2,\]
which contradicts the strict inequality. 
\end{proof}
We now use this non-local characterization in order to prove the last result concerning the $\gamma^{1/2}$-convexity/concavity. 
\begin{prop}
A strictly monotonic function $f$ is (strictly) $\gamma^{1/2}$-convex if and only if  $f^{-1}$ is (strictly) $\gamma^{1/2}$-concave. 
\label{inverse}
\end{prop}
\begin{proof}
By applying the formula for the derivative of an inverse function, we have, for any $x\neq y$, 
\[{f^{-1}}'(x){f^{-1}}'(y)=\frac{1}{f'(f^{-1}(x))f'(f^{-1}(y)).}\]
If $f$ is $\gamma^{1/2}$-convex, then, according to the non-local characterization,
\begin{align*}
f'(f^{-1}(x))f'(f^{-1}(y))\leq \biggl(\frac{f(f^{-1}(x))-f(f^{-1}(y))}{f^{-1}(x)-f^{-1}(y)}\biggr)^2 =\biggl(\frac{x-y}{f^{-1}(x)-f^{-1}(y)}\biggr)^2.
\end{align*}
Hence, 
\[{f^{-1}}'(x){f^{-1}}'(y)\geq \biggl(\frac{f^{-1}(x)-f^{-1}(y)}{x-y}\biggr)^2. \]
We prove the converse with the same method.  
\end{proof}
Let us finally come to the proof of our main theorem (\ref{theorem atmostbistable}), namely: 
\begin{theorem*}
Let $f, g \in C^3(\mathbb{R}_+, \mathbb{R}_+)$ be two functions such that $f'>0, g'>0$ or $f'<0, g'<0$ on $\mathbb{R}_+$. 
We recall that system $\eqref{initial system}$ refers to 
\begin{align*}
\begin{cases}
\dot x=\alpha f(y)-x\\
\dot y =\beta g(x)-y
\end{cases}.
\tag{\ref{initial system}}
\end{align*} 
Then
\begin{itemize}
\item If $f$ and $g$ are $\gamma^{1/2}$-convex, and (at least) one of these functions is strictly $\gamma^{1/2}$-convex, then for any $\alpha, \beta >0$ system \eqref{initial system} has at most three equilibria, and is either monostable or bistable. 
\item If $f$ and $g$  are $\gamma^{1/2}$-concave,  and (at least) one of these functions is strictly $\gamma^{1/2}$-concave, then for any $\alpha, \beta >0$, system \eqref{initial system} is monostable. 
\end{itemize}
\end{theorem*}

\begin{proof}
\begin{itemize}
\item As evidenced by the examples, composing with an affine function does not change the $\gamma^{1/2}$-convexity/concavity of the functions. Hence, according to Proposition  \ref{compose}, under the first hypothesis, $x\mapsto \alpha f(\beta g(x))$ is strictly $\gamma^{1/2}$-convex. According to Proposition \ref{Psi functions}, this function has at most three fixed points, which shows that system \eqref{initial system} has at most three equilibria. 

Moreover, it is well-known that the basins of attraction of stable equilibrium points are open sets. Indeed, if $\bar x_s$ is a stable equilibrium of a differential equation, then, by definition, there exists an open set $U$ which contains $\bar x_s$ and which is included in the basin of attraction. Let $x_0$ be a point of this basin. By denoting $\Phi_T$ the flow associated to this ODE at time $T$, we get $\Phi_T(x_0)\in U$ for $T$ large enough, and $\Phi_T^{-1}(U)$ is thus a neighborhood of $x_0$ which is included in the basin of attraction of $\bar x_s$. 

Hence the system cannot be tristable, and is thus at most bistable. 
\item Under this second hypothesis, $x\mapsto \alpha f(\beta g (x))$ is $\gamma^{1/2}$-concave. According to Proposition~\ref{Psi functions}, this proves that $x\mapsto \alpha f(\beta g(x))$ has a unique equilibrium, and thus that system \eqref{initial system} is monostable. 
\end{itemize}
\end{proof}

\subsection{An extension to cyclic $n$-dimensional ODEs}
As we have seen, the property satisfied by the $\gamma^{1/2}$-convex/concave functions provides an attractive way to check that the function has at most three equilibria. This result can be used to ensure that some ODE models of higher-dimension and of cyclic nature are at most bistable. 

We consider the ODE system
\begin{align}
\begin{cases}
\dot x_1&=f_1(x_n)-x_1\\
\dot x_2&=f_2(x_1)-x_2\\
\vdots\\
\dot x_n&=f_n(x_{n-1})-x_n
\end{cases},
\label{system n variables}
\end{align}
where $f_1, \ldots, f_n \in C^1(\mathbb{R}_+, \mathbb{R}_+)$ are some strictly monotonic functions.

As for the 2-dimensional case, we note that $(\bar{x}_1, \bar{x}_2,\ldots, \bar{x}_n)$ is an equilibrium of \eqref{system n variables} if and only if 
\begin{align*}
\begin{cases}
\bar{x}_1=f_1(\bar{x}_n)\\
\bar{x}_2=f_2(\bar{x}_1) \\
\vdots\\
\bar{x}_{n-1}=f_{n-1}(\bar{x}_{n-2})\\
\bar{x}_n=f_n\circ f_{n-1}\circ \ldots \circ f_1(\bar x_n).
\end{cases}
\end{align*}
Hence, the number of equilibria of~\eqref{system n variables} is equal to the number of fixed points of $f_n\circ f_{n-1}\circ \ldots\circ f_1$. Thus, if $f_n\circ f_{n-1}\circ \ldots\circ f_1$ is decreasing, \textit{i.e.}, if there is an odd number of decreasing functions among $f_1, \ldots, f_n$, then \eqref{system n variables} has a unique equilibria. Otherwise, the previous theoretical framework naturally leads to the following proposition: 

\begin{prop}
Let us assume that $f_1,.\ldots, f_n \in C^3(\mathbb{R}_+, \mathbb{R}_+)$ are monotonic functions such that for all $i\in \{1,\ldots, n\}$ $\lvert f_i'\rvert>0$ on $\mathbb{R}_+$, and that exactly an even number of these functions are decreasing. 

If these functions are all $\gamma^{1/2}$-convex (resp. $\gamma^{1/2}$-concave) and at least one of them is strictly $\gamma^{1/2}$-convex (resp. $\gamma^{1/2}$-concave), then \eqref{system n variables} has a most three (resp. one) equilibria. Moreover, under the weak hypothesis 
\[ \text{for all $x$ fixed point of $f_1 \circ  \ldots\circ f_n$, \quad $(f_1 \circ  \ldots\circ f_n)'(x)\neq 1$},\] 
system~\eqref{system n variables} is at most bistable. 
\end{prop}

\begin{proof}
According to Propositions \ref{Psi functions} and \ref{compose}, $f_n\circ\ldots \circ f_1$ has at most three equilibria. Moreover, if it has three equilibria, then, according to Proposition \ref{prop fixed_point_derivative}, there exists $\bar x:=(\bar{x}_1, \ldots, \bar{x}_n)$ such that 
\[f_1'(\bar{x}_n)f_2'(\bar{x}_1) \ldots f_n'(\bar{x}_{n-1})\geq 1.\]
The Jacobian matrix of the right-hand side of \eqref{system n variables} in $\bar{x}$ is 
\begin{align*}
J_{\bar x}=
\begin{pmatrix}
-1 & 0 & \cdots & 0 &   f_1'(\bar{x}_n)\\
f_2'(\bar{x}_1) & -1 & 0 &  \cdots & 0\\
0 & f_3'(\bar{x}_2) & -1 & \ddots& \vdots\\
\vdots & \ddots & \ddots & \ddots & 0\\
0 & 0 &0  & f_n'(\bar{x}_{n-1}) & -1
\end{pmatrix}.
\end{align*}
Thus, its characteristic polynomial is 
\[(X+1)^n-f_1'(\bar{x}_n)f_2'(\bar{x}_1)\ldots f_n'(\bar{x}_{n-1}).\]
Then $\bigl( f_1'(\bar{x}_n)f_2'(\bar{x}_1)\ldots f_n'(\bar{x}_{n-1})\bigr)^{1/n}-1>0$ is an eigenvalue of $J_{\bar x}$, which proves that this point is asymptotically unstable. Hence, system \eqref{system n variables} has indeed at most two stable equilibria. 
\end{proof}
Let us remark that, as soon as $n\geq 3$, we cannot a priori assert that any solution will converge to some equilibrium point. In fact, as shown in \cite{buse2010dynamical} for $n=3$, some trajectories may be periodic.

\section{Determining the parameters for mono/bistability}
\label{Section4}

In the previous section, we have shown that, under general hypotheses on $f$ and $g$, the system
\begin{align*}
\begin{cases}
\dot x=\alpha f(y)-x\\
\dot y =\beta g(x)-y
\end{cases}
\tag{\ref{initial system}}
\end{align*}
is either monostable or bistable. The purpose of this section is, once the functions $f, g$ are fixed, to determine which parameters $\alpha, \beta$ induce a monostable system, and which ones induce a bistable system. 

To answer this question, we introduce a general framework that will encompass this particular result. Using this framework, we will also recover the result of the third section, which states that, under some hypotheses, system \eqref{initial system} cannot have more than two stable equilibria. 

\subsection{General framework}

We introduce a family of sets: for $n\in \mathbb{N}^*$, let
$$E_n:=\biggl\{(x,y)\in (\mathbb{R}_+^{n})^2 :\, \forall i\neq j,  \frac{x_i}{f(y_i)}=\frac{x_j}{f(y_j)}, \frac{y_i}{g(x_i)}=\frac{y_j}{g(x_j)}, \, x_i\neq x_j\biggr\}. $$
Notice that, since $f$ and $g$ are one-to-one, $E_n$ can also be written  
$$E_n=\biggl\{(x,y)\in (\mathbb{R}_+^{n})^2 :\, \forall i\neq j,  \frac{x_i}{f(y_i)}=\frac{x_j}{f(y_j)}, \frac{y_i}{g(x_i)}=\frac{y_j}{g(x_j)},\, y_i\neq y_j\biggr\}. $$
By convention, we take $E_1=\mathbb{R}_+^2$.
We also introduce two subsets of the above sets $E_n$:
\[E_n^>:=\biggl\{(x,y)\in E_n : \underset{i\in \{1,\ldots,n\}}\min\biggl(\frac{x_ig'(x_i)}{g(x_i)}\frac{y_if'(y_i)}{f(y_i)} \biggr)>1 \biggr\}\]
\[E_n^\geq:=\biggl\{(x,y)\in E_n : \underset{i\in \{1,\ldots,n\}}\min\biggl(\frac{x_ig'(x_i)}{g(x_i)}\frac{y_if'(y_i)}{f(y_i)} \biggr)\geq1 \biggr\}\]
We note that, in the particular case $n=1$, $E_1^\geq $ is the closure of $E_1^>$.

We also introduce 
\begin{align*}
G_n : E_n &\longrightarrow \mathbb{R}_+^2\\
 (x,y)&\longmapsto \biggl(\frac{x_i}{f(y_i)}, \frac{y_i}{g(x_i)}\biggr), 
\end{align*}
for some $i\in \{1,\ldots, n\}$. Remark that by the definition of $E_n$, the function $G_n$ is well-defined since the choice of $i$ is arbitrary.

The following property explains how the sets $E_n^>$ and $E_n^\geq$ and the function $G_n$ can be used to study the number of equilibria of \eqref{initial system}. \\
\begin{prop}
Let $n\in \mathbb{N}^*$. 
\begin{itemize}
\item If $(\alpha, \beta)\in G_n(E_n^>)$, then system $\eqref{initial system}$ has at least $2n+1$ equilibria. 
\item If system $\eqref{initial system}$ has $2n$ equilibria or more, then $(\alpha, \beta)\in G_n({E_n^\geq})$ .
\end{itemize}
\label{E_n}
\end{prop}
Here is a direct and useful consequence of this proposition:
\begin{cor}
If $E_n^>=\emptyset$, then for any choice of $\alpha$, $\beta$, system \eqref{initial system} has at most $2n-1$ equilibria.  
\label{cor empty}
\end{cor}






\begin{proof}[Proof of Proposition \ref{E_n}]
Let us assume that $E_n^>\neq \emptyset$, and let $(\alpha, \beta) \in G_n(E_n^>)$. Then, there exists $(x, y) \in E_n$ such that, for all $i\in \{1,\ldots, n\}, $
\[\frac{ x_ig'(x_i)}{g( x_i)}\frac{ y_if'(y_i)}{f(y_i)}>1,\quad
\text{ and } \quad
\alpha=\frac{x_i}{f(y_i)}, \quad \beta=\frac{y_i}{g(x_i)}.
\]
In other words, all the points $(x_i, y_i)$ are  equilibria of the system, and satisfy
$\alpha \beta f'(y_i)g'(x_i)>1.$
According to Corollary \ref{cor fixed points f positive},  this proves that $x\mapsto \alpha f(\beta g(x))$ has at least $2n+1$ fixed points, and thus that  system \eqref{initial system} has at least $2n+1$ equilibria. 

Now, let us assume that system \eqref{initial system} has at least $2n$ equilibria. According to Proposition \ref{prop fixed_point_derivative},   there exist at least $n$ of these equilibria (denoted $(\bar x_1, \bar y_1), \ldots, (\bar x_n, \bar y_n))$ which satisfy 
$\alpha \beta f'(\bar y_i)g'(\bar x_i)\geq 1.$
Since each $(\bar x_i, \bar y_i)$ is an equilibrium of \eqref{initial system}, then for all $i\in\{1,\cdots,n\}$,
\begin{align*}
\alpha =\frac{\bar x_i}{f(\bar y_i)}, \quad \beta =\frac{\bar y_i}{g(\bar x_i)}, 
\end{align*}
which proves that $(\alpha, \beta)\in G_n\bigl({E_n^\geq}\bigr).$
\end{proof}

\subsection{Application to $n=1$: determining the parameters for mono/bistability}

\subsubsection{General statement}
In \cite{cherry2000make}, Cherry and Adler state conditions on $f$, $g$ which ensure that, for some values of $\alpha, \beta$, system \eqref{initial system} is multistable. As seen in Section \ref{Section2}, under the hypotheses of Theorem \ref{theorem atmostbistable}, this system is at most bistable, which means that  multistability actually simplifies to mere bistability. 

In the following theorem, we improve this result in order to identify, depending on $f$ and $g$, the spaces of parameters $(\alpha, \beta)$ for which the system is monostable or bistable. 
From now on, $\overline{A}$  denotes the closure of a set $A \subset \R^2$, and $A^C$ its complement. 
\begin{theorem}
Let us define 
\[E_1^>:=\biggl\{(x,y)\in \mathbb{R}_+^2 : \biggl(\frac{xg'(x)}{g(x)}\frac{yf'(y)}{f(y)} \biggr)>1 \biggr\}
\text{ and } 
G_1:
(x,y)\longmapsto \biggr(\frac{x}{f(y)}, \frac{y}{g(x)}\biggl) \text{ for }(x,y)\in \R_+^2.
\]
Let us assume that $f$ and $g$ are both $\gamma^{1/2}$-convex, and that at least one of them is strictly $\gamma^{1/2}$-convex. We have the following alternative:
\begin{itemize}
\item If $(\alpha, \beta)\in G_1(E_1^>)$, then \eqref{initial system} has exactly three equilibria, among which exactly two  stable equilibria.
\item If  $(\alpha, \beta)\in \overline{G_1(E_1^>)}^C$, then \eqref{initial system} has a unique equilibrium, which is stable. 
\end{itemize}
\label{theo three-one}
\end{theorem}

\begin{proof}
This is a direct consequence of Proposition \ref{E_n}, applied with $n=1$.  
\end{proof}
\begin{rem}
We note that, if $E_1^>=\emptyset$, \textit{i.e.}, if 
\[\underset{x>0}{\sup}\biggl(\bigg\lvert\frac{xg'(x)}{g(x)}\; \bigg\rvert\biggr) \; \underset{y>0}{\sup}\biggl(\bigg\lvert\frac{yf'(y)}{f(y)}\bigg\rvert\biggr)<1,\]
then \eqref{initial system} is monostable for any $\alpha, \beta$, which is nothing but the theorem of Cherry and Adler in~\cite{cherry2000make}. 
\end{rem}

In general, this criterion does not lead to an explicit expression of $G_1(E_1^>)$, but it does for $E_1^>$. Hence, the set $G_1(E_1^>)$ is numerically easily derived.
These computations may be significantly facilitated by assuming that the system is symmetric, \textit{i.e.}, of the form 
\begin{align}
\begin{cases}
\dot x = \alpha f(y)-x\\
\dot y = \alpha f(x)-y.
\end{cases}
\label{sym system}
\end{align}
In this very special case, we are led to studying the fixed points of $x\mapsto \alpha f(\alpha f(x))$. Since any fixed point of $\alpha f$ is also a fixed point of $\alpha f \circ \alpha f$, this system has at least one `diagonal' equilibrium, \textit{i.e.}, an equilibrium of the form $(\bar x, \bar x)$. The search for fixed points is thus made much simpler, and we obtain  the following proposition.
\begin{prop}
Let us define
\[E_s^>:=\biggl\{x\in \mathbb{R}_+ : \left\lvert \frac{xf'(x)}{f(x)}\right\rvert >1\biggr\}\]
and 
\begin{align*}
G_s:E_s^>&\longrightarrow \mathbb{R}_+\\
x&\longmapsto \frac{x}{f(x)} .
\end{align*}
$G_s$ is increasing if $f$ is decreasing, and decreasing if $f$ is increasing. Moreover,
\begin{itemize}
\item If $\alpha \in G_s(E_s^>)$, then system \eqref{sym system} has three equilibria, and is bistable.
\item If $\alpha \in \overline{G_s(E_s^>)}^C$, then system \eqref{sym system} is monostable. 
\end{itemize}
\label{sym proposition}
\end{prop}

Using Theorem \ref{theo three-one}, this result can be easily proven: 
\begin{proof}
If $f$ is strictly decreasing, then $G_s$ is the product of two  positive and increasing functions, and is thus increasing. If $f$ is strictly increasing, since the derivative of $G_s$ is equal to 
\[x\mapsto \frac{f(x)-xf'(x)}{f(x)^2},\]
the results holds according to the definition of $E_s^>$. 

Let us now prove the second part of the proof. 
First, let us note that $(\alpha,\alpha) \in G_1(E_1^>)$ if and only if $\alpha\in G_s(E_s^>)$. Indeed, if $(\alpha,\alpha) \in G_1(E_1^>)$, then there exists $(x,y)\in E_1^>$ such that 
\[(\alpha,\alpha)=G_1(x,y).\]
Since $G_1(x,y)=G_1(y,x)$ for all $(x,y)$ and $G_1$ is injective, then $x=y$, which proves that 
\[(\alpha,\alpha)=G_1(x,x)=(G_s(x,x), G_s(x,x)).\]
The converse inclusion is obvious from this last equality. 
We conclude by applying Theorem~\ref{theo three-one}. \\
\end{proof}

\begin{rem}
In the case where the system is symmetric and cooperative (\textit{i.e.}, $f$ is increasing in \eqref{sym system}), all the equilibria are exactly the points of the form $(\bar{x}, \bar{x})$, where $\bar{x}$ is a root of $f_\alpha : x\mapsto \alpha f(x)$. 
Indeed, let us assume that there exists $\bar x$ a root of $f_\alpha \circ f_\alpha$ which is not a root of $f_\alpha$. Then, $\bar y :=f_\alpha(\bar x)\neq \bar x$ is also a root of  $f_\alpha \circ f_\alpha$. Let us assume that $\bar y<\bar x$ (the case $\bar x <\bar y$ is similar). Since $f_\alpha$ is increasing, $\bar x = f_\alpha(\bar y)<f_\alpha(\bar{x})=\bar y.$
This leads to a contradiction, which proves the result. 

On the contrary, if the system is competitive, there exists at most one `diagonal equilibrium', since $\alpha f$ has at most one fixed point if $f$ is decreasing. 
\end{rem}

\subsubsection{Examples with Hill functions}

We now illustrate the previous two theorems, in the case where $f$ and $g$ are two Hill functions (shifted or not), \textit{i.e.}, that $f$ and $g$ are defined on $\mathbb{R}_+$ by
\begin{align*}
f:z&\longmapsto \frac{1+\lambda_1 (\frac{z}{z_{01}})^a}{1+(\frac{z}{z_{01}})^a} &g:z&\longmapsto \frac{1+\lambda_2 (\frac{z}{z_{02}})^{b}}{1+(\frac{z}{z_{02}})^{b}},
\end{align*} 
with $\lambda_1, \lambda_2 \in [0,1)$, $a,b\geq 1$, $z_{01}, z_{02}>0.$

First, we note that, for all $\alpha, \beta>0$, 

\[
\begin{cases}
\alpha f(y)-x=0\\
\beta g(x)-y=0
\end{cases} \quad 
\text{ if and only if } \quad
\begin{cases}
\alpha\tilde{f}(\frac{y}{z_{01}})-x=0\\
\beta\tilde{g}(\frac{x}{z_{02}})-y=0
\end{cases},
\]
where 
\[
\tilde f:z\longmapsto \frac{1+\lambda_1 z^{a}}{1+z^{a}} \quad \text{ and } \quad \tilde g:z\longmapsto \frac{1+\lambda_2 z^{b}}{1+z^{b}}
\]
and that 
\[\alpha \beta f'(y)g'(x)=\bigl(\frac{\alpha}{z_{01}}\bigr)\bigl(\frac{\beta}{z_{02}}\bigr)\tilde{f}'\bigl(\frac{y}{z_{01}}\bigr)\tilde{g}'\bigl(\frac{x}{z_{02}}\bigr),\]
which shows that we can choose, without loss of generality, $z_{01}=z_{02}=1$. \\

For conciseness, we present the result only in the case where $\lambda_1=\lambda_2=0$: the complete result for general values of $\lambda_1, \lambda_2$ as well as the proof can be found in Appendix~\ref{AppA}. 

\begin{prop}
Let us assume that $\lambda_1=\lambda_2=0$ and let  $\rho:=ab$. 
If $\rho<1$, then $E_1^>=\emptyset$. \\
If $\rho\geq1$, we define
\begin{align*}
x^-:&=\biggl(\frac{1}{\rho-1}\biggr)^{1/b},\quad y^-:=\biggl(\frac{1}{\rho-1}\biggr)^{1/a},\\
r^-(x):&=\biggl(\frac{1+x^{b}}{(\rho-1)x^{b}-1}\biggr)^{1/a}, \quad s^-(y):=\biggl(\frac{1+y^{a}}{(\rho-1)y^{a}-1}\biggr)^{1/b}.
\end{align*}
Then
\begin{equation}
\begin{split}
E_1^>&=\{(x,y)\in (x^-, +\infty) \times (y^-, +\infty) : y>r^-(x)\}\\
&=\{(x,y)\in (x^-, +\infty) \times (y^-, +\infty) : x>s^-(y)\}.
\end{split}
\label{Hill_system}
\end{equation}
\end{prop}
Together with Theorem \ref{theo three-one}, this allows us to identify where, in the parameter space defined by $(\alpha, \beta)$, we have bistability or monostability.
These results are illustrated by Figure~\ref{fig:paramspace} for various choices of $a$, $b$, $\lambda_1$ and $\lambda_2$.
  
\begin{figure}[h!]
\centering
\includegraphics[width=0.4\textwidth]{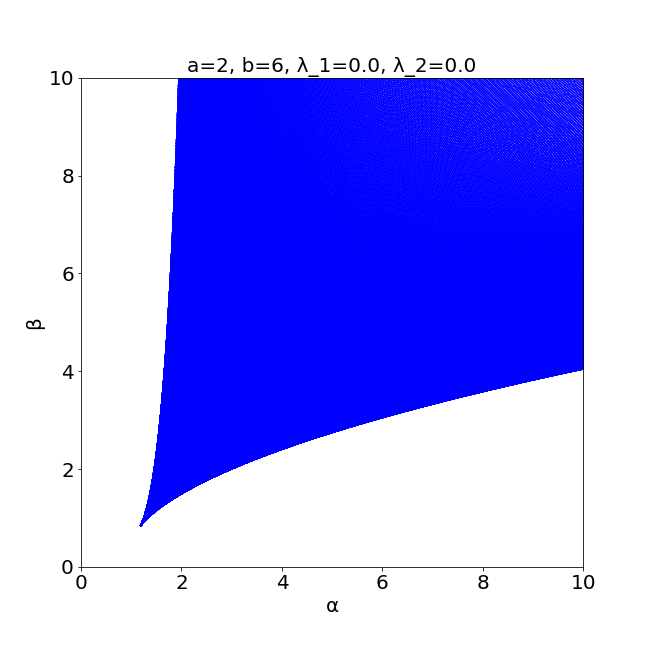}
\includegraphics[width=0.4\textwidth]{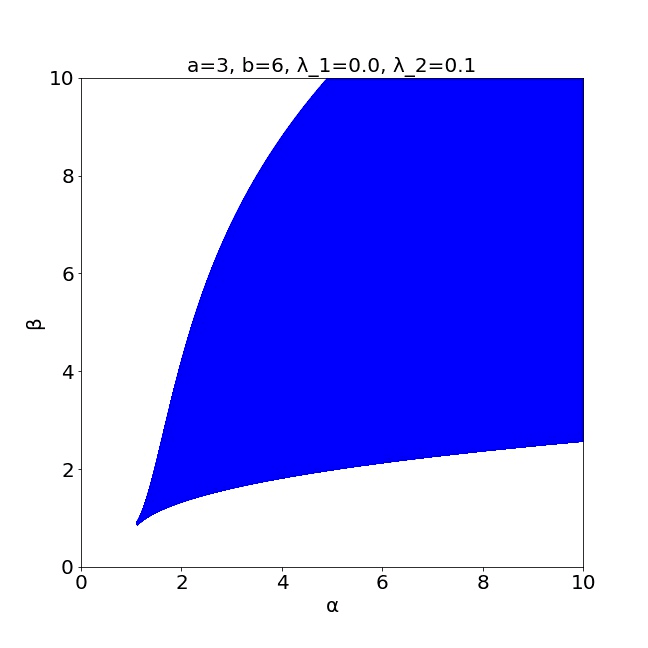}
\includegraphics[width=0.4\textwidth]{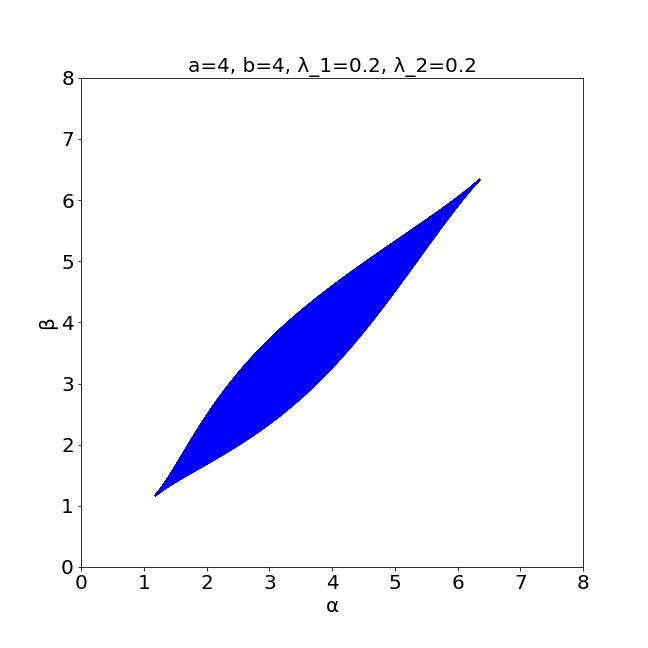}
\includegraphics[width=0.4\textwidth]{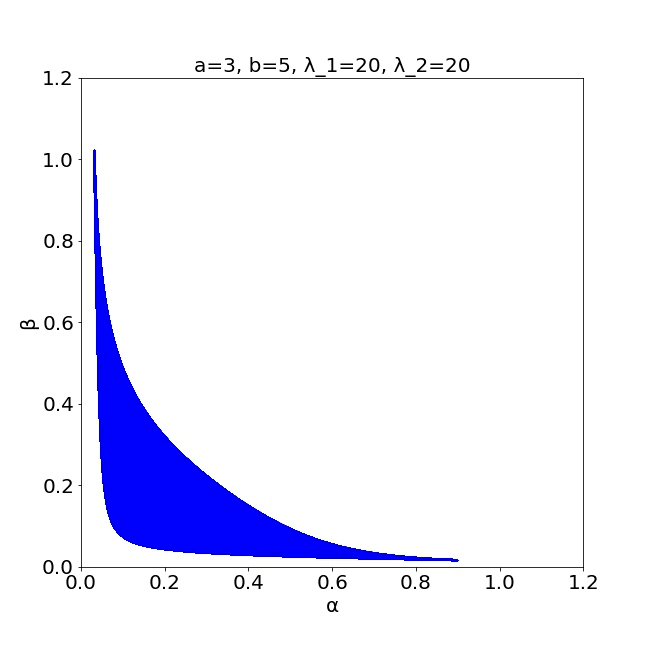}
\caption{The parameter spaces for system \eqref{initial system} with Hill functions. The blue fields represent the parameters $(\alpha, \beta)$ for which the system is bistable, and the white one the parameters for which it is monostable. The problem remains unsolved only for the values at the boundary of the blue field.  The parameters of the Hill functions are indicated above each graph. \\}\label{fig:paramspace}
\end{figure}

We now study the symmetric case. Still assuming that $f$ is a (shifted) Hill function, \textit{i.e.}, 
\begin{align*}
f :
z\mapsto \frac{1+\lambda (\frac{z}{z_0})^a}{1+(\frac{z}{z_0})^a},
\end{align*}
with $z_0>0$, $a>1$, $\lambda\in \mathbb{R}_+\backslash \{1\}$, we now consider that the system is symmetric. Then, with the notations of Proposition~\ref{sym proposition}, we can identify $E_s$ and $G_s(E_s)$. Again, for clarity of presentation and because this is the most standard case in applications, we assume that $\lambda=0$: the complete result for the other values of $\lambda$ and the proof are gathered in Appendix~\ref{AppB}. 
\begin{prop}
Let us assume that $\lambda=0$ and $a>1$. 
We define: 
$\alpha_0^-:=z_0\bigl(\frac{1}{a-1}\bigr)^{1/a}\frac{a}{a-1}.$
Then,
\[G_s(E_s^>)=(\alpha_0^-, +\infty).\]
\label{Hill_sym }
\end{prop}
Figure \ref{fig:paramspace-sym} illustrates the symmetric case by showing the monostable and bistable regions of the parameter space $(\lambda,\alpha)$, for various choices of $a$. Notice that for $\lambda=0$, the stable region for $\alpha$ is indeed of the form $
(\alpha_0^-, +\infty)$, as shown in the previous proposition.

\begin{figure}[h!]
\centering
\includegraphics[width=0.45\textwidth]{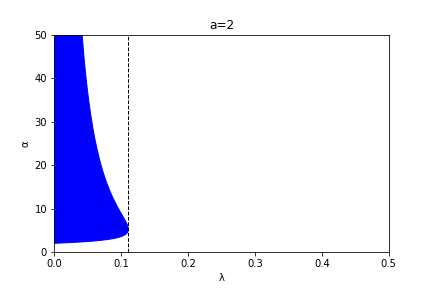}
\includegraphics[width=0.45\textwidth]{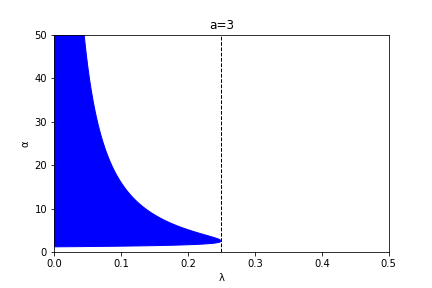}
\includegraphics[width=0.45\textwidth]{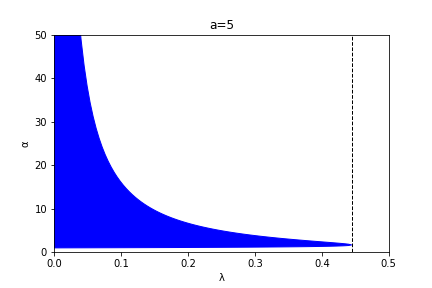}
\includegraphics[width=0.45\textwidth]{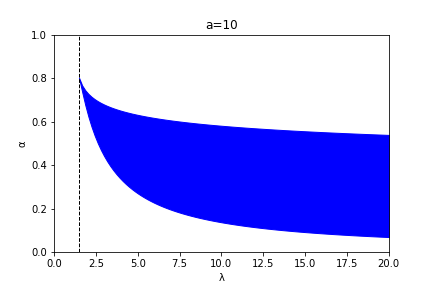}

\caption{The parameter spaces $(\lambda, \alpha)$ is the case of a symmetric system \eqref{sym system} with Hill functions. The blue set stands for the values of $(\lambda, \alpha)$ which ensure bistability, and the white fields represent the  parameters for which this system is monostable. As for the general case, the question of the number of equilibria remains unsolved for the values at the boundary of the blue field. 
All the Hill functions involved have a parameter $z_0$ equal to one; the value of the parameter $a$ is specified above each graph. The vertical dashed line indicates the highest $\lambda$ for which a symmetric system can be bistable (the smallest one on the fourth graphic, which exemplifies the  cooperative case). This value can be explicitly computed, as shown in the Appendix.\\}\label{fig:paramspace-sym}
\end{figure}

\subsection{Application to $n=2$}

We now apply Proposition \ref{E_n} in the specific case where $n=2$, in order to retrieve the result of the second section. 
When $n=2$, we get 
\[E_2=\biggl\{(x_1, x_2, y_1, y_2)\in \mathbb{R}_+^4 : \frac{x_1}{f(y_1)}= \frac{x_2}{f(y_2)}, \frac{y_1}{g(x_1)}=\frac{y_2}{g(x_2)}, \, x_1 \neq x_2\biggr\}, \]
and, since for any $(x_1, x_2, y_1, y_2)\in E_2$, 
\[\frac{x_i}{f(y_i)}= \frac{x_1-x_2}{f(y_1)-f(y_2)} \quad \text{and} \quad \frac{y_i}{g(x_i)}=\frac{y_1-y_2}{g(x_1)-g(x_2)}\]
we can rewrite 
\begin{align*}
E_2^\geq&=\biggl\{(x_1, x_2, y_1, y_2)\in E_2: \frac{\lvert x_1-x_2\rvert }{\lvert g(x_1)-g(x_2)\rvert } \frac{\lvert y_1-y_2\rvert }{\lvert f(y_1)-f(y_2)\rvert }\min\bigl(g'(x_1)f'(y_1), g'(x_2)f'(y_2) \bigr)\geq1 \biggr\}\\
&\subset \biggl\{(x_1, x_2, y_1, y_2)\in \mathbb{R}_+^4  \min\bigl(g'(x_1)f'(y_1), g'(x_2)f'(y_2) \bigr)\geq \frac{\lvert g(x_1)-g(x_2)\rvert}{\lvert x_1-x_2\rvert }\frac{\lvert f(y_1)-f(y_2)\rvert }{\lvert y_1-y_2\rvert } \biggr\}.
\end{align*}
Thus, if $f$ and $g$ are $\gamma^{1/2}$-convex, and one of these functions is strictly $\gamma^{1/2}$-convex, as seen with Proposition~\ref{non-local}, we get,  for any $(x_1, x_2, y_1, y_4)\in \mathbb{R}_+^4$, 
\[f'(y_1)f'(y_2)\leq \biggl(\frac{f(y_2)-f(y_1)}{y_2-y_1}\biggr)^2 \quad \text{and} \quad g'(x_1)g'(x_2)\leq \biggl(\frac{g(x_2)-g(x_1)}{x_2-x_1}\biggr)^2,\]
(and one of this inequality is in fact strict)
and thus
\[g'(x_1)f'(y_1)f'(y_2)g'(x_2)< \biggl(\frac{\lvert g(x_1)-g(x_2)\rvert}{\lvert x_1-x_2\rvert}\frac{\lvert f(y_1)-f(y_2)\rvert }{\lvert y_1-y_2\rvert }\biggr)^2,\]
which proves that 
\[g'(x_1)f'(y_1)<\frac{\lvert g(x_1)-g(x_2)\rvert}{\lvert x_1-x_2\rvert}\frac{\lvert f(y_1)-f(y_2)\rvert }{\lvert y_1-y_2\rvert } \quad \text{or} \quad g'(x_2)f'(y_2)<\frac{\lvert g(x_1)-g(x_2)\rvert}{\lvert x_1-x_2\rvert}\frac{\lvert f(y_1)-f(y_2)\rvert }{\lvert y_1-y_2\rvert}, \]
and therefore that $E_2^\geq=\emptyset$ under this hypothesis. 

%


\appendix
\section{Appendix}

\subsection{Complement of Proposition \ref{Hill_system}} \label{AppA}
Here is the complement and the proof of Proposition \ref{Hill_system}. 

\begin{prop}
We define:
$\phi:\lambda\mapsto\frac{1-\sqrt{\lambda}}{1+\sqrt{\lambda}}$ and 
$\rho:=ab\phi(\lambda_1)\phi(\lambda_2)$.
If $\rho<1$, then $E_1^>=\emptyset$. \\
If $\rho\geq1$, we define:
\begin{itemize}
\item If $\lambda_1>0$ and $\lambda_2>0:$
\begin{align*}
x^\pm:&=\biggl[\frac{1}{\sqrt{\lambda_2}}+\frac{1}{2\lambda_2}(1+\sqrt{\lambda_2})^2\sqrt{\rho-1}\biggl({\sqrt{\rho-1}\pm\sqrt{\rho-\phi({\lambda_2})^2}}\biggr) \biggr]^{1/b}\\
y^\pm:&=\biggl[\frac{1}{\sqrt{\lambda_1}}+\frac{1}{2\lambda_1}(1+\sqrt{\lambda_1})^2\sqrt{\rho-1}\biggl({\sqrt{\rho-1}\pm\sqrt{\rho-\phi(\lambda_1)^2}}\biggr)\biggr]^{1/a}\\
\alpha(x):&=(1+\lambda_1)-(1+\sqrt{\lambda_1})^2(1+\sqrt{\lambda_2})^2\rho\frac{x^{b}}{(1+x^{b})(1+\lambda_2x^{b})}\\
\beta(y):&=(1+\lambda_2)-(1+\sqrt{\lambda_1})^2(1+\sqrt{\lambda_2})^2\rho\frac{y^{a}}{(1+y^{a})(1+\lambda_1y^{a})}\\
r^\pm(x):&=\biggl[\frac{-\alpha(x)\pm\sqrt{\alpha(x)^2-4\lambda_1}}{2\lambda_1}\biggl]^{1/a}\\
s^\pm(y):&=\biggl[\frac{-\beta(y)\pm\sqrt{\beta(y)^2-4\lambda_2}}{2\lambda_2}\biggl]^{1/b}.
\end{align*}
Then, 
\begin{align*}
E_1^>&=\{(x,y)\in (x^-, x^+) \times (y^-, y^+) : r^-(x)<y<r^+(x)\}\\
&=\{(x,y)\in (x^-, x^+) \times (y^-, y^+) : s^-(y)<x<s^+(y)\}.
\end{align*}
\item If $\lambda_1=0$ and $\lambda_2>0:$
\begin{align*}
x^\pm:&=\biggl[\frac{1}{\sqrt{\lambda_2}}+\frac{1}{2\lambda_2}(1+\sqrt{\lambda_2})^2\sqrt{\rho-1}\biggl({\sqrt{\rho-1}\pm\sqrt{\rho-\phi({\lambda_2})^2}}\biggr) \biggr]^{1/b}\\
y^-:&=\bigg(\frac{1}{\rho-1}\bigg)^{1/a}\\
r^-(x):&=\biggl[\frac{(1+x^{b})(1+\lambda_2x^{b})}{-\lambda_2x^{2b}+(ab(1-\lambda_2)-(1+\lambda_2))x^{b}-1}\bigg]^{1/a}\\
\beta(y):&=(1+\lambda_2)-(1+\sqrt{\lambda_1})^2(1+\sqrt{\lambda_2})^2\rho\frac{y^{a}}{(1+y^{a})(1+\lambda_1y^{a})}\\
s^\pm(y):&=\biggl[\frac{-\beta(y)\pm\sqrt{\beta(y)^2-4\lambda_2}}{2\lambda_2}\biggl]^{1/b}
\end{align*}
\begin{align*}
E_1^>&=\{(x,y)\in (x^-, x^+) \times (y^-, +\infty) : y>r^-(x)\}\\
&=\{(x,y)\in (x^-, x^+) \times (y^-, +\infty) : s^-(y)<x<s^+(y)\}.
\end{align*}
\item If $\lambda_1=\lambda_2=0:$
\begin{align*}
x^-&:=\biggl(\frac{1}{\rho-1}\biggr)^{1/b}, \quad y^-:=\biggl(\frac{1}{\rho-1}\biggr)^{1/a}\\
r^-(x)&:=\biggl(\frac{1+x^{n_2}}{(\rho-1)x^{b}-1}\biggr)^{1/a}, \quad  s^-(y):=\biggl(\frac{1+y^{a}}{(\rho-1)y^{a}-1}\biggr)^{1/b}
\end{align*}
\begin{align*}
E_1^>&=\{(x,y)\in (x^-, +\infty) \times (y^-, +\infty) : y>r^-(x)\}\\
&=\{(x,y)\in (x^-, +\infty) \times (y^-, +\infty) : x>s^-(y)\}.
\end{align*}
\label{Hill_system_complete}
\end{itemize}
\end{prop}
In order to prove this proposition, we will need this very simple lemma. 
\begin{lem}
Let $P(X):=aX^2+bX+c$ be a polynomial of degree two.
We assume that $a>0$ and $c>0$. If $P$ has two real roots, then they have the same sign. Moreover, $P$ has two positive roots if and only if $b+2\sqrt{ac}<0$. 
\label{poly}
\end{lem}


\begin{proof}[Proof of Proposition \ref{Hill_system}]
Let $(x,y)\in \mathbb{R}_+^2.$
\[\frac{xg'(x)}{g(x)}\frac{yf'(y)}{f(y)}=ab(1-\lambda_1)(1-\lambda_2)\frac{x^by^a}{(1+x^b)(1+\lambda_2 x^{b})(1+y^{a})(1+\lambda_1 y^{a})}.\]
We denote
\begin{align*}
\omega:&=x^{b}\\
\delta:&=y^{a}\\
C:&=ab(1-\lambda_1)(1-\lambda_2)=(1+\sqrt{\lambda_1})^2(1+\sqrt{\lambda_2})^2\rho\\ 
C_1(\delta):&=(1+\delta)(1+\lambda_1 \delta)\\
C_2(\omega):&=(1+\omega)(1+\lambda_2 \omega)\\
 P(\omega,\delta):&=\lambda_1C_2(\omega)\delta^2+(C_2(\omega)(1+\lambda_1)-C\omega)\delta+C_2(\omega)\\
&=\lambda_2C_1(\delta)\omega^2+(C_1(\delta)(1+\lambda_2)-C\delta)\omega+C_1(\delta).
\end{align*}
Then, 
\begin{align*}
(x,y)\in E_1^> &\iff \frac{xg'(x)}{g(x)}\frac{yf'(y)}{f(y)} \iff P(\omega, \delta)<0.
\end{align*}
We now assume that $\lambda_1>0$ and $\lambda_2>0$: the proof for $\lambda_1=0$ or $\lambda_2=0$ is similar and simpler, as $P$ becomes a polynomial of degree 1 for the variable $\omega$ or $\gamma$. 

According to lemma \ref{poly}, there exist $\omega, \delta>0$ such that $P(\omega, \delta)<0$ if and only if
\begin{align*}
Q_2(\omega)&:=C_2(\omega)(1+\lambda_1)-C\omega+2\sqrt{\lambda_1}C_2(\omega)<0\\
&\iff C_2(\omega)(1+\sqrt{\lambda_1})^2-C\omega<0\\
&\iff\lambda_2\omega^2+(1+\lambda_2-\frac{C}{(1+\sqrt{\lambda_1})^2})\omega+1<0\\
&\iff \lambda_2\omega^2+(1+\lambda_2-(1+\sqrt{\lambda_2})^2\rho)\omega+1<0.
\end{align*}
According to lemma \ref{poly}, there exists $\omega>0$ such that $Q_2(\omega)<0$ if and only if
\begin{equation*}
(1+\sqrt{\lambda_2})^2-(1+\sqrt{\lambda_2})^2\rho<0 \iff \rho>1.
\end{equation*}
Therefore, $Q_2(\omega)<0$ if and only if $\rho>1$ and $\omega \in (\omega^-, \omega^+)$, where $\omega^-, \omega^+$ (which depend on $n_1$, $n_2$, $\lambda_1$ and $\lambda_2$), denote the two positive roots of $Q_2$. 

Lastly, $P(\omega,\delta)<0$ if and only if $\rho>1$, $\omega \in (\omega^-, \omega^+)$, and $\delta \in ({\tilde r}^-(\omega), {\tilde r}^+(\omega))$, where, for all $\omega \in (\omega^-, \omega^+)$, ${\tilde r}^-(\omega), {\tilde r}^+(\omega)$ denote the two positive roots of $\delta \mapsto P(\omega, \delta)$.

With the same reasoning, by denoting $\delta_-, \delta_+$ the two positive roots of $Q_1(\delta):=C_1(\delta)(1+\lambda_2)-C\delta+\sqrt{\lambda_2}C_1(\delta)<0$, and, for all $\delta \in (\delta^-, \delta^+)$, ${\tilde s}^-(\delta)$ and ${\tilde s}^+(\delta)$, the two positive roots of $\omega \mapsto P(\omega, \delta)$, $P(\omega,\delta)<0$ if and only if $\rho>1$, $\delta \in (\delta^-, \delta^+)$, and $\omega \in ({\tilde s}^-(\delta), {\tilde s}^+(\delta))$.

We denote
\begin{align*}
x^-:&=(\omega^-)^{1/b},  x^+:=(\omega^+)^{1/b}\\
y^-:&=(\delta^-)^{1/a}, y^+:=(\delta^+)^{1/a}\\
r^-(x):&=({\tilde r}^-(\omega))^{1/a}, r^+(\omega):=({\tilde r}^+(\omega))^{1/a}\\
s^-(y):&=({\tilde s}^-(\delta))^{1/b}, r^+(\delta):=({\tilde r}^+(\delta))^{1/b}.
\end{align*}
The values of $x^-, x^+, y^-, y^+, r^+, r^-, s^-$ and $s^+$ are directly computed from the formula of $Q_1$, $Q_2$ and $P$. 
\end{proof}

\subsection{Complement of Proposition \ref{Hill_sym }}\label{AppB}

\begin{prop}
Let us denote 
$\displaystyle\lambda_0^-:=\biggl(\frac{a-1}{a+1}\biggr)^2$ and  $\displaystyle \lambda_0^+=\biggl(\frac{a+1}{a-1}\biggr)^2$.
\begin{itemize}
\item If $\lambda\in (\lambda_0^-, \lambda_0^+)$, then $E_s=\emptyset$. In other words, system \eqref{sym system} is monostable, for any value of $g$. 
\item If $\lambda=0$, we define: 
$\displaystyle \alpha_0^-:=z_0\bigl(\frac{1}{a-1}\bigr)^{1/a}\frac{a}{a-1}.$ 
Then,
$E_s^>=(\alpha_0^-, +\infty).$
\item If $\lambda \in  (0, \lambda_0^-]$, we define:
\begin{align*}
\omega_\lambda^-&:=\frac{a-1-(a+1)\bigl(\lambda+\sqrt{(1-\lambda)(\lambda_0^--\lambda)}\bigr)}{2\lambda} \\
\omega_\lambda^+&:=\frac{a-1-(a+1)\bigl(\lambda-\sqrt{(1-\lambda)(\lambda_0^--\lambda)}\bigr)}{2\lambda}\\
\alpha_\lambda^-&:=z_0(\omega_\lambda^-)^{1/a}\frac{1+\omega_\lambda^-}{1+\lambda \omega_\lambda^-} , \qquad
\alpha_\lambda^+:=z_0(\omega_\lambda^+)^{1/a}\frac{1+\omega_\lambda^+}{1+\lambda \omega_\lambda^+}.
\end{align*}
Then, 
$G_s(E_s^>)=(\alpha_\lambda^-, \alpha_\lambda^+). $
\item If $\lambda\geq \lambda_0^+$, we define:
\begin{align*}
\tau_\lambda^-&:=\frac{-(a+1)+(a-1)\bigl(\lambda-\sqrt{(\lambda-1)(\lambda-\lambda_0^+)}\bigr)}{2\lambda} \\
\tau_\lambda^+&:=\frac{-(a+1)+(a-1)\bigl(\lambda+\sqrt{(\lambda-1)(\lambda-\lambda_0^+)}\bigr)}{2\lambda}\\
\beta_\lambda^-&:=z_0(\tau_\lambda^+)^{1/a}\frac{1+\tau_\lambda^+}{1+\lambda \tau_\lambda^+} , \qquad
\beta_\lambda^+:=z_0(\tau_\lambda^-)^{1/a}\frac{1+\tau_\lambda^-}{1+\lambda \tau_\lambda^-}.
\end{align*}
Then, 
$G_s(E_s^>)=(\beta_\lambda^-, \beta_\lambda^+). $
\label{Hill_sym_complete}
\end{itemize}
\end{prop}

\begin{proof}

With the same reasoning as for the asymmetric system, we can assume that $z_0=1$. 
Let $x\in \mathbb{R}_+$. We denote:
$\omega:=x^a $ and 
\[ P(\omega):=(1+\lambda \omega)(1+\omega)-n\lvert \lambda-1\rvert \omega
 =\lambda \omega^2+(\lambda+1-a\lvert \lambda-1\rvert)\omega+1.
\]
Since 
\[\bigg\lvert \frac{xf'(x)}{f(x)} \bigg\rvert=\frac{a\lvert \lambda-1\rvert x^a}{(1+\lambda x^a)(1+x^a)},\]
we clearly have:
$
x\in E_s^> \iff \big\lvert\frac{xf'(x)}{f(x)}\big\rvert>1
\iff P(\omega)<0.
$ 

If $\lambda=0$, then:
$P(\omega)=1-(a-1)\omega.$
Thus,
$
P(\omega)<0 \iff \omega >\frac{1}{a-1}
\iff x>(\frac{1}{a-1})^{1/a}$.\\
Since $G_s(\bigl(\frac{1}{a-1}\bigr)^{1/a})=(\frac{1}{a-1})^{1/a}\frac{a}{a-1}:=a_0^-$, the result follows.

We now assume that $\lambda>0$. Denoting $\Delta$ the discriminant of P, we get
\[
\Delta=(\lambda+1)^2+a^2(\lambda-1)^2-2a(\lambda+1)\lvert \lambda-1\rvert -4\lambda
=\lvert \lambda-1\rvert \bigl(\lvert \lambda-1\rvert (1+a^2)-2a(\lambda+1)\bigr).
\]
We distinguish two cases:
\begin{itemize}
\item If $\lambda<1$, recalling that $\lambda_0^-=(\frac{a-1}{a+1})^2$, we have
\[
\Delta = (1-\lambda)\bigl((1-a)^2-\lambda (1+a)^2\bigr)
=-(1-\lambda)(1+a)^2(\lambda-\lambda_0^-).
\]
Hence, if $\lambda>\lambda_0$, then $P(\omega)>0$, and thus $x \notin E_s^>$.

If $\lambda \in (0,\lambda_0]$, then $P$ has two roots, and we have 
\[P(\omega)<0 \iff \omega \in (\omega_\lambda^-, \omega_\lambda^+) \iff  x \in ((\omega_\lambda^-)^{1/a}, (\omega_\lambda^+)^{1/a}).\]
Hence, 
$E_s^>=((\omega_\lambda^-)^{1/a}, (\omega_\lambda^+)^{1/a}).$
Since $G_s$ is increasing, we get 
\[
G_s(E_s)=(G_s((\omega_\lambda^-)^{1/a}), G_s(\omega_\lambda^+)^{1/a})
=(\alpha_\lambda^-, \alpha_\lambda^+).
\]
\item If $\lambda>0$, the reasoning is exactly the same. In this case
\[
\Delta = (\lambda-1)\bigl(\lambda (a-1)^2-(a+1)^2\bigr)
=(\lambda-1)(a-1)^2\bigl(\lambda-\lambda_0^+\bigr),
\]
where $\lambda_0^+ = (\frac{a+1}{a-1})^2$.
We get 
$E_s^>=((\tau_\lambda^-)^{1/a}, (\tau_\lambda^+)^{1/a})$,
and, lastly, since $G_s$ is decreasing in this case, 
\[
G_s(E_s^>)=(G(\tau_\lambda^+,), G(\tau_\lambda^-))
=(\beta_\lambda^-, \beta_\lambda^+), 
\]
which ends the proof. 
\end{itemize}

\end{proof}

\subsection{Other sigmoid functions which are $\gamma^{1/2}$-convex }
\label{AppC}
The following sigmoid functions are strictly $\gamma^{1/2}$-convex:
\begin{enumerate}[(i)]
\item The logistic function, and more generally, the general logistic functions, \textit{i.e.}, all the functions of the shape 
\[x\mapsto \biggl(\frac{1}{1+e^{-x}}\biggr)^\alpha, \quad \alpha>0.\]
\item The hyperbolic tangent: 
\[x\mapsto \frac{e^x-e^{-x}}{e^x+e^{-x}}.\]
\item The arctan function.
\item The Gudermannian function: 
\[x\mapsto \int_0^x{\frac{1}{\cosh(t)}dt}.\]
\item The error function 
\[x\mapsto \frac{2}{\sqrt{\pi}}\int_0^x{e^{-t^2}dt}.\]
\end{enumerate}

\begin{proof}
\begin{enumerate}[(i)]
\item We note that the result is immediate if we assume $\alpha\geq1$, because the function happens to be the composite of $\gamma^{1/2}$-convex functions in this case. Surprisingly, the results holds true if $\alpha<1$. 
This result is equivalent to showing that $f:x\mapsto\frac{-1}{\alpha}(1+e^x)^{-\alpha}$ is strictly $\gamma^{1/2}$-convex, for any $\alpha>0$ (the result is in fact true for any $\alpha \neq 1$). 
We indeed compute 
\[f'(x)=e^x(1+e^x)^{-(\alpha+1)}, \quad f''(x)=\frac{1-\alpha e^x}{1+e^x}f'(x), \quad f^{(3)}(x)=\frac{\alpha e^{2x}-(3\alpha+1)e^x+1}{(1+e^x)^2}f'(x)^2, \]
and find
\[f^{(3)}(x)f'(x)-\frac{3}{2}f''(x)^2=-\frac{f'(x)^2}{2(1+e^x)}(\alpha^2e^{2x}+2e^x+1)<0.\]
\item We rewrite
\[\frac{e^x-e^{-x}}{e^x+e^{-x}}= \frac{(e^{x})^2-1}{(e^x)^2+1}.\]
The hyberbolic tangent is thus the composite function of the  exponential function, the square function and a homographic function, which proves that it is strictly $\gamma^{1/2}$-convex as the composition of $\gamma^{1/2}$-convex functions, one of which is strictly $\gamma^{1/2}$-convex. 
\item It is equivalent to proving that the tangent function is strictly $\gamma^{1/2}$-concave on $(-\pi/2, \pi/2)$. Since $\tan'=\frac{1}{\cos^2}$, the result follows according to the positivity and the concavity of the cosine function on this interval. 
\item Since this function may be rewritten  as
\[x\mapsto \int_0^x{\frac{1}{\cosh(t)}dt}=2\arctan(\tanh(\frac{x}{2})), \] 
this result is true according to the $\gamma^{1/2}$-convexity of arctan and tanh. 
\item The result is obvious. 
\end{enumerate}
\end{proof}


\bibliographystyle{abbrv}
\bibliography{Biblio_article}

\end{document}